\providecommand{\AddToHook}[2]{}
\documentclass[10pt,onefignum,onetabnum,hidelinks]{siamart251216}

\usepackage{amssymb}
\usepackage{mathtools}
\usepackage{graphicx}
\usepackage{subcaption}
\usepackage{booktabs}
\usepackage{multirow}
\usepackage{siunitx}
\usepackage{comment}
\usepackage{algorithm}
\usepackage{algpseudocode}

\newsiamremark{remark}{Remark}

\headers{SAV with Pullback Corrections}{S. Zhang and J. Shen}

\title{SAV Schemes with Decomposition-Induced Pullback Corrections for Gradient Flows}

\author{Shiheng Zhang\thanks{Purdue University, West Lafayette, IN, USA (\email{shzhang3722@gmail.com}).}
\and Jie Shen\thanks{Eastern Institute of Technology, Ningbo, China
(\email{jshen@eitech.edu.cn}). Corresponding author. The work of this author is supported in part by NSFC grants 12531015 and 12371409.}}

\begin{document}
\maketitle

\begin{abstract}
The scalar auxiliary variable (SAV) method is a widely used framework for constructing linear, unconditionally energy-stable time discretizations of gradient flows.  In a first-order SAV step, eliminating the auxiliary variable shows that the state equation is a semi-implicit update augmented by a rank-one positive semidefinite correction generated by the previous nonlinear force.  The multiple-SAV (MSAV) method produces the corresponding correction componentwise and therefore yields a positive semidefinite correction whose rank is bounded by the number of energy components.  This algebraic viewpoint separates two mechanisms that are usually coupled in MSAV: the number of scalar variables used to track the nonlinear energy and the rank of the correction applied to the state equation.  We introduce a pullback-corrected SAV (PB-SAV) family that keeps a single scalar auxiliary variable for the total nonlinear energy but replaces the SAV rank-one correction by the pullback correction induced by an admissible component decomposition.  The correction remains positive semidefinite, has rank at most the number of components, and may be changed from one time step to the next without changing the scalar auxiliary variable.  We prove modified-energy dissipation laws for fixed and step-dependent decompositions, establish first-order convergence after a fixed spatial discretization under standard smoothness and positivity assumptions, derive a refinement identity in which the gain from refining a decomposition is an explicit weighted variance, and give a Sherman--Morrison--Woodbury implementation for the resulting low-rank perturbation of the standard semi-implicit solve.  We also show, in finite dimensions, that the pullback correction is the Gauss--Newton matrix of a least-squares representation of the nonlinear energy.  Numerical experiments for finite-dimensional gradient flows, Allen--Cahn dynamics, a finite-rank nonlocal gradient flow, and a linear nonlocal Cahn--Hilliard model illustrate regimes in which PB-SAV and SAV share the same first-order behavior and differ only in the error constant, and regimes in which PB-SAV reduces the trajectory error by a large factor.
\end{abstract}



\section{Introduction}
\label{sec:introduction}

Gradient flows are a basic mathematical model for dissipative evolution.  They appear in phase-field models for interfacial motion, nonlocal and local materials models, fluid and complex-fluid systems, image processing, and optimization.  In all of these settings the solution is driven by the variational derivative of an energy, and the decrease of that energy is a structural property of the continuous problem.  A time discretization that loses this structure can generate artificial energy growth, distort metastable dynamics, or require a step-size restriction that is not intrinsic to the underlying gradient flow.  This is why the construction of schemes with a discrete energy law has become a central issue in the numerical analysis of gradient-flow problems.

We consider gradient flows in the form
\begin{equation}\label{eq:intro-gf}
  \partial_t \phi = \mathcal G \frac{\delta E}{\delta \phi},
\end{equation}
where $E[\phi]$ is the free energy and $\mathcal G$ is a self-adjoint nonpositive operator.  The continuous energy identity is
\begin{equation*}
  \frac{\mathrm d}{\mathrm dt}E[\phi(t)]
  = \left(\frac{\delta E}{\delta \phi},\mathcal G\frac{\delta E}{\delta \phi}\right) \le 0.
\end{equation*}
For computation one commonly writes
\begin{equation}\label{eq:intro-splitting}
  E[\phi]=\frac12(\phi,\mathcal L\phi)+E_1[\phi],
\end{equation}
where the self-adjoint positive semidefinite operator $\mathcal L$ is treated implicitly and the nonlinear part $E_1$ is treated explicitly or semi-explicitly.  The challenge is to combine three properties that often pull in different directions: unconditional energy stability, good trajectory accuracy, and a linear solve whose dominant part is independent of the time step number.

Many approaches have been developed to preserve energy dissipation at the discrete level; we refer to \cite{MR4378430,shen2019new} for representative reviews.  The scalar auxiliary variable (SAV) method \cite{shen2018scalar,shen2019new} is especially attractive because it converts a broad class of nonlinear gradient flows into linearly implicit schemes that dissipate a modified energy.  The method introduces
\[
    q(t):=\sqrt{E_1[\phi(t)]+C},
\]
with $C$ chosen so that $E_1[\phi]+C>0$, and evolves the pair $(\phi,q)$ in a way that is equivalent to the original gradient flow when $q(0)=\sqrt{E_1[\phi(0)]+C}$.  First-order and higher-order SAV schemes are linear, unconditionally stable with respect to a modified energy, and can often be implemented using constant-coefficient solvers.  Convergence and error estimates for SAV schemes were developed in \cite{shen2018convergence}.  The multiple scalar auxiliary variable (MSAV) method \cite{cheng2018multiple} extends this idea by introducing one scalar variable for each term in a decomposition $E_1=\sum_\alpha E_{1,\alpha}$, which is useful when an energy contains components of different scales or physical origins.

The SAV framework has also led to a large literature of variants and extensions.  These include auxiliary-variable variants \cite{hou2019variant,huang2020highly}, exponential SAV and SAV-exponential-integrator schemes \cite{liu2020exponential,ju2022generalized}, arbitrarily high-order constructions \cite{gong2019arbitrarily}, Lagrange multiplier and generalized SAV formulations \cite{cheng2020new,cheng2021generalized}, and general linear energy-stable frameworks \cite{tan2022general}.  Another line of work, including relaxed and energy-optimized SAV methods \cite{jiang2022improving,zhang2022generalized,huang2024computationally,liu2024novel,zhang2025relaxed}, addresses the consistency between the numerical auxiliary variable $q^n$ and its defining value $\sqrt{E_1[\phi^n]+C}$.  The present paper addresses a different question: how much of the componentwise correction produced by MSAV can be obtained while retaining only one scalar auxiliary variable?

The motivation for the present work is easiest to see after the auxiliary variables have been eliminated.  In the first-order SAV scheme \eqref{eq:s2-sav-discrete-a}--\eqref{eq:s2-sav-discrete-b}, the state variable $\phi^{n+1}$ does not simply satisfy the usual semi-implicit update; it is supplemented by a positive semidefinite correction in the single direction of the total nonlinear force.  More precisely, substitution of the scalar update gives a linear equation for $\phi^{n+1}$ containing
\[
    \mathcal B_1^n\psi
    =
    \frac{1}{2(Q^n)^2}\,(U^n,\psi)\,U^n,
    \qquad
    U^n:=\frac{\delta E_1}{\delta\phi}[\phi^n],
    \qquad
    Q^n:=\sqrt{E_1[\phi^n]+C}.
\]
This correction is rank one: all information in the nonlinear force is compressed into the single vector $U^n$.  Such compression is natural when the nonlinear energy behaves as one coherent term, but it can be too restrictive when $E_1$ contains several terms with different scales, different spatial locations, or different physical meanings.

The MSAV approach keeps more of this componentwise information.  If
\[
    E_1[\phi]+C
    =
    \sum_{\alpha=1}^{r}\bigl(E_{1,\alpha}[\phi]+C_\alpha\bigr),
    \qquad
    E_{1,\alpha}[\phi]+C_\alpha>0,
\]
then eliminating the $r$ component auxiliary variables in the first-order MSAV scheme \eqref{eq:s2-msav-cont-a}--\eqref{eq:s2-msav-cont-b} produces the correction
\[
    \mathcal B_r^n\psi
    =
    \sum_{\alpha=1}^{r}
    \frac{1}{2(Q_\alpha^n)^2}\,(U_\alpha^n,\psi)\,U_\alpha^n,
    \qquad
    U_\alpha^n:=\frac{\delta E_{1,\alpha}}{\delta\phi}[\phi^n],
    \qquad
    Q_\alpha^n:=\sqrt{E_{1,\alpha}[\phi^n]+C_\alpha}.
\]
The range of $\mathcal B_r^n$ lies in the span of the component forces $U_1^n,\ldots,U_r^n$, so the correction can represent several independent force directions instead of only their sum.  This observation separates two effects that are coupled in MSAV: the use of several scalar auxiliary variables and the use of a higher-rank correction in the state equation.  The central idea of this paper is that the second effect can be obtained without the first.

We therefore introduce a pullback-corrected SAV (PB-SAV) scheme.  PB-SAV retains the single scalar variable $q$ associated with the total nonlinear energy $E_1[\phi]+C$, so the modified energy and scalar update remain those of standard SAV.  The only change is in the eliminated state equation: the rank-one SAV correction $\mathcal B_1^n$ is replaced by the component correction $\mathcal B_r^n$. PB-SAV thus keeps one scalar auxiliary variable, the scalar modified-energy law, and the low-rank perturbation of the semi-implicit solve of SAV, while using the componentwise correction of MSAV.

The term \emph{pullback} describes this metric viewpoint.  Let
\[
    R(\phi):=\bigl(Q_1(\phi),\ldots,Q_r(\phi)\bigr),
    \qquad
    Q(\phi)=\|R(\phi)\|_2.
\]
Then
\[
    \mathcal B_r^n=2\,DR(\phi^n)^*DR(\phi^n),
    \qquad
    \mathcal B_1^n=2\,DQ(\phi^n)^*DQ(\phi^n).
\]
Thus standard SAV first collapses the component map $R$ to the scalar norm $Q$ and then forms a rank-one pullback correction.  PB-SAV instead pulls back the Euclidean metric of the full component space before this collapse.  The one-component decomposition recovers standard SAV, whereas finer decompositions add positive semidefinite directions, can be changed adaptively from step to step, and can improve trajectory accuracy without introducing componentwise auxiliary variables.

The main contributions are as follows.  First, we construct the first-order PB-SAV family, prove a modified-energy dissipation law with the same scalar modified energy as SAV, and establish first-order convergence after a fixed spatial discretization under standard smoothness and positivity assumptions.  Second, we show that the pullback correction dominates the rank-one SAV correction in the positive semidefinite order; the gap is an explicit weighted variance of normalized component directional derivatives.  This identity also gives a monotonicity result under refinement of a component partition.  Third, we allow the decomposition to change from step to step and prove that the modified-energy law is unchanged, because the auxiliary variable tracks only the total nonlinear energy.  Fourth, we derive a Sherman--Morrison--Woodbury solver form for the low-rank perturbation of the usual semi-implicit operator.  Finally, we give a finite-dimensional optimization viewpoint in which the PB-SAV correction is the Gauss--Newton matrix of a least-squares representation of the nonlinear energy, and we use numerical examples to identify when the higher-rank correction is mainly a change in error constant and when it provides a substantial trajectory improvement.

The paper is organized as follows.  Section~\ref{sec:sav-msav} reviews the SAV and MSAV reformulations and fixes notation.  Section~\ref{sec:rank-r-sav} derives the pullback correction operators, introduces PB-SAV, proves the stability, convergence, and refinement results, and describes the low-rank solver and step-dependent decompositions.  Section~\ref{sec:optimization-viewpoint} gives the finite-dimensional optimization interpretation and the relaxation of the auxiliary variable that preserves the modified-energy law.  Section~\ref{sec:numerics} presents numerical experiments for finite-dimensional problems, Allen--Cahn dynamics, a finite-rank nonlocal gradient flow, and a linear nonlocal Cahn--Hilliard model.  Section~\ref{sec:conclusion} concludes.  We restrict the analysis in this paper to first-order time discretizations.

\section{Gradient flows and SAV preliminaries}
\label{sec:sav-msav}

\subsection{Gradient flow and energy splitting}
\label{subsec:gradient-flow-splitting}

We first recall the standard gradient flow notation. Let $\phi(t)$ satisfy
\begin{equation}
\frac{\partial \phi}{\partial t}
=
\mathcal G \mu,
\qquad
\mu=\frac{\delta E}{\delta\phi},
\label{eq:s2-general-gradient-flow}
\end{equation}
where $(\cdot,\cdot)$ denotes the standard $L^2$ inner product on the spatial domain and $\mathcal G$ is a self-adjoint, negative semidefinite linear operator. Then the energy dissipates according to
\begin{equation}
\frac{d}{dt}E[\phi(t)]
=
\left(
\frac{\delta E}{\delta\phi},
\mathcal G\frac{\delta E}{\delta\phi}
\right)
\le 0.
\label{eq:s2-continuous-energy-law}
\end{equation}

We use the splitting
\begin{equation}
E[\phi]=\tfrac12(\phi,\mathcal L\phi)+E_1[\phi],
\label{eq:s2-energy-splitting}
\end{equation}
where $\mathcal L$ is self-adjoint and positive semidefinite. We write
\begin{equation}
U[\phi]
:=
\frac{\delta E_1}{\delta\phi}
\label{eq:s2-U-definition}
\end{equation}
for the nonlinear force, so that $\delta E/\delta\phi=\mathcal L\phi+U[\phi]$. Throughout, $E_1$ is assumed to be bounded from below, and $C$ is chosen so that $E_1[\phi]+C>0$ on the admissible set. Let
\(
Q(\phi):=\sqrt{E_1[\phi]+C}.
\)
The shift $C$ affects the auxiliary variable but not the variational derivative $\delta E_1/\delta\phi$.

\subsection{SAV reformulation and first-order scheme}
\label{subsec:first-order-sav}

The SAV reformulation augments the state by one scalar auxiliary variable $q(t)$ and writes the gradient flow as the coupled system
\begin{align}
\frac{\partial \phi}{\partial t}
&=
\mathcal G\left(
\mathcal L\phi+\frac{q}{Q(\phi)}U[\phi]
\right),
\label{eq:s2-sav-cont-a}
\\
\frac{dq}{dt}
&=
\frac{1}{2Q(\phi)}
\left(
U[\phi],\frac{\partial \phi}{\partial t}
\right).
\label{eq:s2-sav-cont-b}
\end{align}
The reformulation is equivalent to the original gradient flow when the auxiliary variable is initialized consistently. Indeed, if $q(0)=Q(\phi(0))$, then
\[
q(t)^2-E_1[\phi(t)]-C
\]
is constant in time. Hence $q(t)=Q(\phi(t))$ for all \(t\), and
\eqref{eq:s2-sav-cont-a} is equivalent to the original gradient flow
\eqref{eq:s2-general-gradient-flow}.

For the splitting \eqref{eq:s2-energy-splitting}, the natural SAV modified energy is
\begin{equation}
\widetilde E[\phi,q]
=
\tfrac12(\phi,\mathcal L\phi)+q^2.
\label{eq:s2-modified-energy-scalar}
\end{equation}
Taking the inner product of \eqref{eq:s2-sav-cont-a} with
$\mathcal L\phi+(q/Q(\phi))U[\phi]$ and using \eqref{eq:s2-sav-cont-b} gives
\begin{equation}
\frac{d}{dt}\left(\tfrac12(\phi,\mathcal L\phi)+q^2\right)
=
\left(
\mathcal L\phi+\frac{q}{Q(\phi)}U[\phi],
\mathcal G\left(
\mathcal L\phi+\frac{q}{Q(\phi)}U[\phi]
\right)
\right)
\le0.
\label{eq:s2-sav-cont-energy-law}
\end{equation}

Let $Q^n:=Q(\phi^n)$ and $U^n:=U[\phi^n]$. The first-order SAV time discretization treats the linear term $\mathcal L\phi$ implicitly and freezes the nonlinear force at the previous time level:
\begin{align}
\frac{\phi^{n+1}-\phi^n}{\Delta t}
&=
\mathcal G\left(
\mathcal L\phi^{n+1}
+
\frac{q^{n+1}}{Q^n}U^n
\right),
\label{eq:s2-sav-discrete-a}
\\
\frac{q^{n+1}-q^n}{\Delta t}
&=
\frac{1}{2Q^n}
\left(
U^n,\frac{\phi^{n+1}-\phi^n}{\Delta t}
\right).
\label{eq:s2-sav-discrete-b}
\end{align}
\begin{proposition}
\label{prop:sav-energy-law}
The scheme \eqref{eq:s2-sav-discrete-a}--\eqref{eq:s2-sav-discrete-b} satisfies
\begin{align}
&\frac{1}{\Delta t}
\left[
\widetilde E[\phi^{n+1},q^{n+1}]
-
\widetilde E[\phi^n,q^n]
\right]
+
\frac{1}{\Delta t}
\left[
(q^{n+1}-q^n)^2
+
\tfrac12\bigl(\phi^{n+1}-\phi^n,\mathcal L(\phi^{n+1}-\phi^n)\bigr)
\right]
\nonumber\\
&\qquad
=
\left(
\mathcal L\phi^{n+1}+\frac{q^{n+1}}{Q^n}U^n,\;
\mathcal G\left(
\mathcal L\phi^{n+1}+\frac{q^{n+1}}{Q^n}U^n
\right)
\right)
\le0.
\label{eq:s2-sav-discrete-law}
\end{align}
\end{proposition}

\begin{proof}
Set
\(
\delta\phi^{n+1}:=\phi^{n+1}-\phi^n
\)
and
\(
\mu^{n+1}:=\mathcal L\phi^{n+1}+(q^{n+1}/Q^n)U^n
\).
Taking the inner product of \eqref{eq:s2-sav-discrete-a} with $\mu^{n+1}$ gives
\[
\frac1{\Delta t}(\delta\phi^{n+1},\mu^{n+1})
=
(\mu^{n+1},\mathcal G\mu^{n+1}).
\]
The left-hand side contains
\[
(\mathcal L\phi^{n+1},\delta\phi^{n+1})
=
\tfrac12(\phi^{n+1},\mathcal L\phi^{n+1})
-
\tfrac12(\phi^n,\mathcal L\phi^n)
+
\tfrac12(\delta\phi^{n+1},\mathcal L\delta\phi^{n+1}),
\]
where self-adjointness of $\mathcal L$ is used.  From the scalar update
\eqref{eq:s2-sav-discrete-b},
\[
\frac{q^{n+1}}{Q^n}(U^n,\delta\phi^{n+1})
=
2q^{n+1}(q^{n+1}-q^n)
=
(q^{n+1})^2-(q^n)^2+(q^{n+1}-q^n)^2.
\]
Substitution gives \eqref{eq:s2-sav-discrete-law}.  The final inequality follows from the assumed nonpositivity of $\mathcal G$.
\end{proof}

The only implicit state term is the linear term $\mathcal L\phi^{n+1}$. Hence \eqref{eq:s2-sav-discrete-a}--\eqref{eq:s2-sav-discrete-b} is linear in $(\phi^{n+1},q^{n+1})$. Eliminating $q^{n+1}$ gives a linear system for $\phi^{n+1}$ whose principal operator is the time-independent semi-implicit operator $I-\Delta t\,\mathcal G\mathcal L$.

\subsection{MSAV with a component decomposition}
\label{subsec:msav}

Suppose that $E_1$ admits a decomposition
\begin{equation}
E_1[\phi]+C
=
\sum_{\alpha=1}^r
\left(E_{1,\alpha}[\phi]+C_\alpha\right),
\qquad
E_{1,\alpha}[\phi]+C_\alpha>0.
\label{eq:s2-component-split}
\end{equation}
A decomposition satisfying the positivity condition in \eqref{eq:s2-component-split} will be called \emph{admissible}. The shifts $C_\alpha$ make the square roots below well defined. They do not change the component forces $U_\alpha=\delta E_{1,\alpha}/\delta\phi$, but they do appear in the correction through the weights $(Q_\alpha)^2$. Define
\begin{equation}
Q_\alpha(\phi)
:=
\sqrt{E_{1,\alpha}[\phi]+C_\alpha},
\qquad
U_\alpha[\phi]
:=
\frac{\delta E_{1,\alpha}}{\delta\phi},
\qquad
\alpha=1,\dots,r.
\label{eq:s2-msav-q}
\end{equation}
MSAV assigns one auxiliary variable to each component and uses
\begin{align}
\frac{\partial \phi}{\partial t}
&=
\mathcal G\left(
\mathcal L\phi
+
\sum_{\alpha=1}^r
\frac{q_\alpha}{Q_\alpha(\phi)}
U_\alpha[\phi]
\right),
\label{eq:s2-msav-cont-a}
\\
\frac{dq_\alpha}{dt}
&=
\frac{1}{2Q_\alpha(\phi)}
\left(
U_\alpha[\phi],\frac{\partial\phi}{\partial t}
\right),
\qquad
\alpha=1,\dots,r.
\label{eq:s2-msav-cont-b}
\end{align}
With consistent initial data $q_\alpha(0)=Q_\alpha(\phi(0))$ for all $\alpha$, the MSAV system is equivalent to the original gradient flow.

The corresponding first-order MSAV scheme treats $\mathcal L\phi$ implicitly and freezes each component force explicitly:
\begin{align}
\frac{\phi^{n+1}-\phi^n}{\Delta t}
&=
\mathcal G\left(
\mathcal L\phi^{n+1}
+
\sum_{\alpha=1}^r
\frac{q_\alpha^{n+1}}{Q_\alpha^n}
U_\alpha^n
\right),
\label{eq:s2-msav-discrete-a}
\\
\frac{q_\alpha^{n+1}-q_\alpha^n}{\Delta t}
&=
\frac{1}{2Q_\alpha^n}
\left(
U_\alpha^n,\frac{\phi^{n+1}-\phi^n}{\Delta t}
\right),
\qquad
\alpha=1,\dots,r.
\label{eq:s2-msav-discrete-b}
\end{align}
Let \(\mathbf q=(q_1,\dots,q_r)\). The scheme satisfies the corresponding
modified-energy law with
\[
\widetilde E[\phi,\mathbf q]
=
\tfrac12(\phi,\mathcal L\phi)
+
\sum_{\alpha=1}^r q_\alpha^2.
\]
The proof is the same as the proof of Proposition~\ref{prop:sav-energy-law}, applied component by component.

\section{The pullback-corrected SAV scheme}
\label{sec:rank-r-sav}
The distinction between SAV and MSAV has two parts: MSAV uses more scalar auxiliary variables, and it also changes the correction that appears after those variables are eliminated.  We now isolate the second effect, which is the basis for PB-SAV.

\subsection{Pullback correction operators from SAV and MSAV}
\label{subsec:variational-form}

For SAV, substitute the scalar update \eqref{eq:s2-sav-discrete-b} into the state equation \eqref{eq:s2-sav-discrete-a}.  The eliminated state equation is
\begin{equation}
\frac{\phi^{n+1}-\phi^n}{\Delta t}
=
\mathcal G\left(
\mathcal L\phi^{n+1}
+
\frac{q^n}{Q^n}U^n
+
\mathcal B^n_1(\phi^{n+1}-\phi^n)
\right),
\label{eq:s3-scalar-eliminated}
\end{equation}
where
\begin{equation}
\mathcal B^n_1\psi
:=
\frac{1}{2(Q^n)^2}
\bigl(U^n,\psi\bigr)U^n.
\label{eq:s3-B-sav}
\end{equation}
Then
\[
(\psi,\mathcal B^n_1\psi)
=
\frac{(U^n,\psi)^2}{2(Q^n)^2}\ge0,
\qquad
\operatorname{range}(\mathcal B^n_1)\subseteq\operatorname{span}\{U^n\}.
\]

Suppose $E_1$ admits the admissible component decomposition \eqref{eq:s2-component-split}  with
\[
Q_\alpha^n:=\sqrt{E_{1,\alpha}[\phi^n]+C_\alpha},
\qquad
U_\alpha^n:=\frac{\delta E_{1,\alpha}}{\delta\phi}[\phi^n].
\]
For MSAV, the same elimination applied to the component updates \eqref{eq:s2-msav-discrete-b} and the state equation \eqref{eq:s2-msav-discrete-a} gives
\begin{equation}
\frac{\phi^{n+1}-\phi^n}{\Delta t}
=
\mathcal G\left(
\mathcal L\phi^{n+1}
+
\sum_{\alpha=1}^r
\frac{q_\alpha^n}{Q_\alpha^n}U_\alpha^n
+
\mathcal B^n_r(\phi^{n+1}-\phi^n)
\right),
\label{eq:s3-msav-eliminated}
\end{equation}
where
\begin{equation}
\mathcal B^n_r\psi
:=
\sum_{\alpha=1}^r
\frac{1}{2(Q_\alpha^n)^2}
\bigl(U_\alpha^n,\psi\bigr)U_\alpha^n,
\label{eq:s3-B-r}
\end{equation}
and 
\[
(\psi,\mathcal B^n_r\psi)
=
\sum_{\alpha=1}^r \frac{(U_\alpha^n,\psi)^2}{2(Q_\alpha^n)^2}\ge0,
\qquad
\operatorname{range}(\mathcal B^n_r)
\subseteq
\operatorname{span}\{U_1^n,\dots,U_r^n\}.
\]

\begin{remark}[Pullback form of the correction]
\label{rem:pullback}
The terminology comes from the component square-root map
\(R(\phi):=\bigl(Q_1(\phi),\dots,Q_r(\phi)\bigr)\),
for which the admissible decomposition gives
\(Q(\phi)^2=E_1[\phi]+C=\sum_{\alpha=1}^r Q_\alpha(\phi)^2=\|R(\phi)\|_2^2\).
Since \(DQ_\alpha(\phi^n)\psi=(U_\alpha^n,\psi)/(2Q_\alpha^n)\), where
\(D\) denotes the Fr\'echet derivative and \(^{*}\) its Hilbert-space adjoint, the correction
\eqref{eq:s3-B-r} is
\[
\mathcal B_r^n=2\,DR(\phi^n)^*DR(\phi^n),
\]
the Euclidean metric on the component variables pulled back to the state variable.
Similarly, the rank-one correction \eqref{eq:s3-B-sav} is
\(\mathcal B_1^n=2\,DQ(\phi^n)^*DQ(\phi^n)\). In finite dimensions,
\(\mathcal B_r^n\) is the Gauss--Newton matrix associated with the least-squares
energy \(\|R(\phi)\|_2^2\).
\end{remark}

\begin{lemma}
\label{lem:B-ordering}
\label{lem:weighted-variance-gap}
Assume \eqref{eq:s2-component-split}. Let
\[
Q^n:=\sqrt{E_1[\phi^n]+C},
\qquad
U^n:=\frac{\delta E_1}{\delta\phi}[\phi^n]
=\sum_{\alpha=1}^r U_\alpha^n .
\]
For any direction \(\psi\), define
\[
y_\alpha
:=
\frac{(U_\alpha^n,\psi)}{(Q_\alpha^n)^2},
\qquad
\bar y
:=
\frac{\sum_{\alpha=1}^r (Q_\alpha^n)^2 y_\alpha}
{\sum_{\alpha=1}^r (Q_\alpha^n)^2}.
\]
Then
\begin{equation}
\left(\psi,(\mathcal B_r^n-\mathcal B_1^n)\psi\right)
=
\frac12
\sum_{\alpha=1}^r
(Q_\alpha^n)^2
\left(y_\alpha-\bar y\right)^2
\ge 0 .
\label{eq:s3-weighted-variance-gap}
\end{equation}
In particular,
\[
\mathcal B_r^n\succeq\mathcal B_1^n,
\]
where \(A\succeq B\) means that \(A-B\) is positive semidefinite. Moreover, the
two corrections agree in the direction \(\psi\) if and only if
\(y_\alpha=\bar y\) for every \(\alpha\).
\end{lemma}

\begin{proof}
By the definition of \(\mathcal B_r^n\),
\[
(\psi,\mathcal B_r^n\psi)
=
\frac12
\sum_{\alpha=1}^r
\frac{(U_\alpha^n,\psi)^2}{(Q_\alpha^n)^2}
=
\frac12
\sum_{\alpha=1}^r
(Q_\alpha^n)^2 y_\alpha^2 .
\]
Since \(U^n=\sum_{\alpha=1}^r U_\alpha^n\) and
\((Q^n)^2=\sum_{\alpha=1}^r (Q_\alpha^n)^2\),
\[
(\psi,\mathcal B_1^n\psi)
=
\frac12
\frac{\left(\sum_{\alpha=1}^r (U_\alpha^n,\psi)\right)^2}
{\sum_{\alpha=1}^r (Q_\alpha^n)^2}
=
\frac12
\left(\sum_{\alpha=1}^r (Q_\alpha^n)^2\right)\bar y^2 .
\]
Therefore
\[
\left(\psi,(\mathcal B_r^n-\mathcal B_1^n)\psi\right)
=
\frac12
\left[
\sum_{\alpha=1}^r (Q_\alpha^n)^2 y_\alpha^2
-
\left(\sum_{\alpha=1}^r (Q_\alpha^n)^2\right)\bar y^2
\right].
\]
The expression in brackets is the weighted variance identity
\[
\sum_{\alpha=1}^r (Q_\alpha^n)^2 y_\alpha^2
-
\left(\sum_{\alpha=1}^r (Q_\alpha^n)^2\right)\bar y^2
=
\sum_{\alpha=1}^r
(Q_\alpha^n)^2(y_\alpha-\bar y)^2,
\]
which proves \eqref{eq:s3-weighted-variance-gap}. The nonnegativity gives
\(\mathcal B_r^n\succeq\mathcal B_1^n\). Equality in the direction \(\psi\)
holds exactly when all terms in the weighted variance vanish, i.e., when
\(y_\alpha=\bar y\) for every \(\alpha\).
\end{proof}

The same variance identity gives a useful monotonicity statement under refinement of the component partition.

\begin{proposition}
\label{prop:refinement-monotonicity}
Let \eqref{eq:s2-component-split} be a decomposition into \(r\)
components. Let \(\mathcal P_{\rm c}\) and \(\mathcal P_{\rm f}\) be two
partitions of \(\{1,\dots,r\}\), with \(\mathcal P_{\rm f}\) refining
\(\mathcal P_{\rm c}\). For a partition \(\mathcal P\), define
\[
\mathcal B_{\mathcal P}^n\psi
:=
\sum_{G\in\mathcal P}
\frac{1}{2w_G}
\bigl(U_G^n,\psi\bigr)U_G^n,
\qquad
U_G^n:=\sum_{\alpha\in G}U_\alpha^n,
\qquad
w_G:=\sum_{\alpha\in G}(Q_\alpha^n)^2 .
\]
Then
\[
\mathcal B_{\mathcal P_{\rm f}}^n
\succeq
\mathcal B_{\mathcal P_{\rm c}}^n .
\]
More precisely, for every direction \(\psi\),
\begin{equation}
\bigl(\psi,(\mathcal B_{\mathcal P_{\rm f}}^n
-\mathcal B_{\mathcal P_{\rm c}}^n)\psi\bigr)
=
\frac12
\sum_{G\in\mathcal P_{\rm c}}
\sum_{\substack{H\in\mathcal P_{\rm f}\\ H\subseteq G}}
w_H\,(y_H-\bar y_G)^2,
\label{eq:s3-refinement-variance}
\end{equation}
where
\[
y_H:=\frac{(U_H^n,\psi)}{w_H},
\qquad
\bar y_G:=\frac{(U_G^n,\psi)}{w_G}.
\]
\end{proposition}

\begin{proof}
Fix \(\psi\) and a block \(G\in\mathcal P_{\rm c}\). The refined blocks
\(H\in\mathcal P_{\rm f}\) with \(H\subseteq G\) form a partition of \(G\), so
\[
U_G^n=\sum_{H\subseteq G}U_H^n,
\qquad
w_G=\sum_{H\subseteq G}w_H .
\]
Applying Lemma~\ref{lem:B-ordering} to this block gives
\[
\sum_{H\subseteq G}\frac{(U_H^n,\psi)^2}{w_H}
-
\frac{(U_G^n,\psi)^2}{w_G}
=
\sum_{H\subseteq G}w_H\,(y_H-\bar y_G)^2 .
\]
Summing over \(G\in\mathcal P_{\rm c}\) and multiplying by \(1/2\) gives
\eqref{eq:s3-refinement-variance}. The right-hand side is nonnegative, hence
\(\mathcal B_{\mathcal P_{\rm f}}^n
\succeq
\mathcal B_{\mathcal P_{\rm c}}^n\).
\end{proof}

\subsection{The PB-SAV scheme and stability}
\label{subsec:rank-r-first-order}

We now define the PB-SAV update.  The scalar variable remains the single SAV auxiliary variable for the total energy, but the rank-one correction in \eqref{eq:s3-scalar-eliminated} is replaced by the pullback correction $\mathcal B_r^n$.  The scheme is
\begin{align}
\frac{\phi^{n+1}-\phi^n}{\Delta t}
&=
\mathcal G\left(
\mathcal L\phi^{n+1}
+
\frac{q^n}{Q^n}U^n
+
\mathcal B_r^n(\phi^{n+1}-\phi^n)
\right),
\label{eq:s3-rank-r-a}
\\
\frac{q^{n+1}-q^n}{\Delta t}
&=
\frac{1}{2Q^n}
\left(
U^n,\frac{\phi^{n+1}-\phi^n}{\Delta t}
\right).
\label{eq:s3-rank-r-b}
\end{align}
\begin{theorem}
\label{thm:rank-r-stability}
Assume that the component decomposition
\eqref{eq:s2-component-split} is admissible. Then the scheme
\eqref{eq:s3-rank-r-a}--\eqref{eq:s3-rank-r-b} satisfies a modified-energy
dissipation law. More precisely, let
\[
\widetilde E[\phi,q]=\tfrac12(\phi,\mathcal L\phi)+q^2,
\qquad
\mathcal S_r^n:=\mathcal B_r^n-\mathcal B_1^n.
\]
Then $\mathcal S_r^n\succeq0$, and
\begin{align}
&\frac{1}{\Delta t}
\left[
\widetilde E[\phi^{n+1},q^{n+1}]
-
\widetilde E[\phi^n,q^n]
\right]
+
\frac{(q^{n+1}-q^n)^2}{\Delta t}
+
\frac{1}{2\Delta t}\bigl(\phi^{n+1}-\phi^n,\mathcal L(\phi^{n+1}-\phi^n)\bigr)
\nonumber\\
&\qquad
+
\frac{1}{\Delta t}
\left(
\phi^{n+1}-\phi^n,\mathcal S_r^n(\phi^{n+1}-\phi^n)
\right)
=
\left(
\Xi^{n+1},
\mathcal G\Xi^{n+1}
\right)
\le0,
\label{eq:s3-rank-r-energy-law}
\end{align}
where
\[
\Xi^{n+1}
:=
\mathcal L\phi^{n+1}
+
\frac{q^n}{Q^n}U^n
+
\mathcal B_r^n(\phi^{n+1}-\phi^n).
\]
In particular, $\widetilde E[\phi^{n+1},q^{n+1}]
\le \widetilde E[\phi^n,q^n]$.
\end{theorem}

\begin{proof}
Set $\delta\phi^{n+1}=\phi^{n+1}-\phi^n$.  Since $\mathcal B_r^n=\mathcal B_1^n+\mathcal S_r^n$ and
\eqref{eq:s3-rank-r-b} implies
\[
\mathcal B_1^n\delta\phi^{n+1}
=
\frac{q^{n+1}-q^n}{Q^n}U^n,
\]
we may rewrite
\[
\Xi^{n+1}
=
\mathcal L\phi^{n+1}
+
\frac{q^{n+1}}{Q^n}U^n
+
\mathcal S_r^n\delta\phi^{n+1}.
\]
Taking the inner product of \eqref{eq:s3-rank-r-a} with $\Xi^{n+1}$ gives
\[
\frac{1}{\Delta t}
\left(\delta\phi^{n+1},\Xi^{n+1}\right)
=
\left(\Xi^{n+1},\mathcal G\Xi^{n+1}\right).
\]
We now identify the terms on the left-hand side with the modified-energy difference and the nonnegative numerical dissipation terms.  The quadratic part satisfies
\[
\left(\mathcal L\phi^{n+1},\delta\phi^{n+1}\right)
=
\tfrac12(\phi^{n+1},\mathcal L\phi^{n+1})-\tfrac12(\phi^n,\mathcal L\phi^n)+\tfrac12\bigl(\delta\phi^{n+1},\mathcal L\delta\phi^{n+1}\bigr),
\]
with the last term nonnegative. The auxiliary-variable update gives
\[
\frac{q^{n+1}}{Q^n}
\left(U^n,\delta\phi^{n+1}\right)
=
2q^{n+1}(q^{n+1}-q^n)
=
(q^{n+1})^2-(q^n)^2+(q^{n+1}-q^n)^2.
\]
The last term is
\[
\left(
\delta\phi^{n+1},
\mathcal S_r^n\delta\phi^{n+1}
\right).
\]
Combining these identities gives \eqref{eq:s3-rank-r-energy-law}. The
right-hand side is nonpositive because $\mathcal G$ is nonpositive, and
$\mathcal S_r^n\succeq0$ by Lemma~\ref{lem:B-ordering}.
\end{proof}
\begin{remark}
Along a smooth exact trajectory, the correction is multiplied by $\phi^{n+1}-\phi^n=O(\Delta t)$.  Therefore replacing $\mathcal B_1^n$ by $\mathcal B_r^n$ does not change the formal first-order consistency of the underlying SAV discretization.
\end{remark}

\subsection{Solver structure}
\label{subsec:solver-structure}

The PB-SAV step is a low-rank update of the semi-implicit operator used by SAV.  Since $\mathcal L$ is self-adjoint and positive semidefinite, rearranging \eqref{eq:s3-rank-r-a} gives
\begin{equation}
\bigl(I-\Delta t\,\mathcal G\mathcal L-\Delta t\,\mathcal G\mathcal B_r^n\bigr)\phi^{n+1}
=
\phi^n+\Delta t\,\mathcal G\frac{q^n}{Q^n}U^n-\Delta t\,\mathcal G\mathcal B_r^n\phi^n.
\label{eq:solver-rankr-system}
\end{equation}
In coordinates that are orthonormal with respect to the discrete inner product, the pullback correction $\mathcal B_r^n$ defined by \eqref{eq:s3-B-r} factors as
\(\mathcal B_r^n=V_nV_n^T\), where
\[
V_n
:=
\Bigl[\tfrac{U_1^n}{\sqrt 2\,Q_1^n},\ldots,\tfrac{U_r^n}{\sqrt 2\,Q_r^n}\Bigr]
\in\mathbb R^{N\times r},
\]
and $N$ is the dimension of the spatial discretization. The system
\eqref{eq:solver-rankr-system} is therefore a perturbation of rank at most \(r\) of
the usual semi-implicit operator
\[
K:=I-\Delta t\,\mathcal G\mathcal L.
\]

\begin{proposition}
\label{prop:solver-invertibility}
Consider a finite-dimensional spatial discretization that preserves the assumed
self-adjointness and semidefiniteness of \(\mathcal G\) and \(\mathcal L\) in
the discrete inner product. Then, for every \(\Delta t>0\), both
\[
K=I-\Delta t\,\mathcal G\mathcal L
\quad\text{and}\quad
I-\Delta t\,\mathcal G(\mathcal L+\mathcal B_r^n)
\]
are nonsingular. Consequently, the PB-SAV state update is uniquely solvable.
\end{proposition}

\begin{proof}
Let \(\mathcal A\) denote either \(\mathcal L\) or
\(\mathcal L+\mathcal B_r^n\). In both cases \(\mathcal A\) is self-adjoint
and positive semidefinite. If
\[
\bigl(I-\Delta t\,\mathcal G\mathcal A\bigr)v=0,
\]
then \(v=\Delta t\,\mathcal G\mathcal A v\). Taking the discrete inner
product with \(\mathcal A v\) gives
\[
0\le (v,\mathcal A v)
=\Delta t\,(\mathcal G\mathcal A v,\mathcal A v)\le0.
\]
Hence \((v,\mathcal A v)=0\), so \(\mathcal A v=0\). The defining equation
then yields \(v=0\), which proves nonsingularity.
\end{proof}

Writing the right-hand side of \eqref{eq:solver-rankr-system} as
\[
g_n:=\phi^n+\Delta t\,\mathcal G\frac{q^n}{Q^n}U^n
-\Delta t\,\mathcal G V_nV_n^T\phi^n,
\]
Proposition~\ref{prop:solver-invertibility} and the matrix determinant lemma
show that
\[
I_r-V_n^T K^{-1}(\Delta t\,\mathcal G V_n)
\]
is nonsingular. The Sherman--Morrison--Woodbury formula
\cite{golub2013matrix} therefore gives
\begin{equation}
\phi^{n+1}
=
K^{-1}g_n
+
K^{-1}(\Delta t\,\mathcal G V_n)
\bigl[I_r-V_n^T K^{-1}(\Delta t\,\mathcal G V_n)\bigr]^{-1}
V_n^T K^{-1}g_n.
\label{eq:solver-smw}
\end{equation}

The displayed factorization and formula use orthonormal coordinates. More
generally, if the discrete inner product is
\((u,v)_h=u^TMv\) with mass matrix \(M\), then
\[
\mathcal B_r^n=V_nV_n^TM,
\]
and every occurrence of \(V_n^T\) in the right-hand side and in the reduced
matrix in \eqref{eq:solver-smw} is replaced by \(V_n^TM\). The solvability
argument above applies without change in the \(M\)-inner product.

In practice, each time step requires $r+1$ applications of $K^{-1}$: one to
$g_n$ and one to each column of $\Delta t\,\mathcal G V_n$. Since
$K=I-\Delta t\,\mathcal G\mathcal L$ is independent of $n$, a factorization
or fast solver can be reused throughout the computation. At fixed \(r\),
this is the same number of solves with $K$ as for first-order MSAV; PB-SAV
differs in having one scalar auxiliary variable and in allowing the components
to be regrouped from step to step. On periodic domains,
when both \(\mathcal G\) and \(\mathcal L\) are Fourier multipliers---for
example, \(\mathcal L=-\Delta\) with \(\mathcal G=-I\) or
\(\mathcal G=\Delta\)---the action of \(K^{-1}\) is diagonal in the Fourier
basis.

\subsection{First-order convergence after spatial discretization}
\label{subsec:first-order-convergence}

We record a minimal convergence result for a fixed spatial discretization. It
makes precise the formal consistency observation above without introducing
problem-dependent PDE regularity assumptions.

\begin{theorem}
\label{thm:first-order-convergence}
Consider a fixed finite-dimensional spatial discretization and a fixed
admissible component decomposition. Let
\(z(t)=(\phi(t),q(t))\) be the solution of the SAV reformulation
\eqref{eq:s2-sav-cont-a}--\eqref{eq:s2-sav-cont-b} on \([0,T]\), with
\(q(t)=Q(\phi(t))\) and \(z\in C^2([0,T])\). Assume that, in a neighborhood
of the exact trajectory,
\[
Q(\phi)\ge q_*>0,\qquad Q_\alpha(\phi)\ge q_*>0
\quad (\alpha=1,\ldots,r),
\]
and that \(U\) and \(U_\alpha\) are continuously differentiable with
uniformly bounded derivatives. Then there exist \(\Delta t_0>0\) and a
constant \(C_T\), independent of \(\Delta t\), such that for
\(0<\Delta t\le\Delta t_0\), the PB-SAV scheme
\eqref{eq:s3-rank-r-a}--\eqref{eq:s3-rank-r-b}, initialized by
\((\phi^0,q^0)=(\phi(0),q(0))\), satisfies
\[
\max_{0\le n\le T/\Delta t}
\left(\|\phi^n-\phi(t_n)\|+|q^n-q(t_n)|\right)
\le C_T\Delta t .
\]
The constant \(C_T\) may depend on the fixed spatial discretization.
\end{theorem}

\begin{proof}
Write
\[
F_\phi(\phi,q)
:=
\mathcal G\left(\mathcal L\phi+\frac{q}{Q(\phi)}U(\phi)\right),
\qquad
F_q(\phi,q)
:=
\frac{1}{2Q(\phi)}\bigl(U(\phi),F_\phi(\phi,q)\bigr),
\]
so that the SAV reformulation is the finite-dimensional ODE
\(\dot z=F(z):=(F_\phi,F_q)\). Let
\[
\mathcal A(\phi):=\mathcal L+\mathcal B_r(\phi),
\qquad
M_{\Delta t}(\phi):=I-\Delta t\,\mathcal G\mathcal A(\phi).
\]
Writing the PB-SAV state equation in terms of
\(\delta\phi^{n+1}:=\phi^{n+1}-\phi^n\) gives
\[
\delta\phi^{n+1}
=
\Delta t\,M_{\Delta t}(\phi^n)^{-1}F_\phi(\phi^n,q^n),
\]
and the scalar update becomes
\[
q^{n+1}-q^n
=
\frac{1}{2Q(\phi^n)}
\bigl(U(\phi^n),\delta\phi^{n+1}\bigr).
\]
Thus one PB-SAV step can be written as
\[
z^{n+1}=\Psi_{\Delta t}(z^n)
=z^n+\Delta t\,H_{\Delta t}(z^n),
\]
where
\[
H_{\Delta t}(\phi,q)
=
\left(
M_{\Delta t}(\phi)^{-1}F_\phi(\phi,q),
\frac{\bigl(U(\phi),M_{\Delta t}(\phi)^{-1}F_\phi(\phi,q)\bigr)}
{2Q(\phi)}
\right).
\]

The lower bounds on \(Q\) and \(Q_\alpha\), together with the stated
smoothness assumptions, imply that \(\mathcal B_r(\phi)\),
\(\mathcal A(\phi)\), and their first derivatives are uniformly bounded in
a compact neighborhood of the exact trajectory. Hence, for
\(\Delta t\) sufficiently small,
\[
\sup_\phi\|M_{\Delta t}(\phi)^{-1}\|\le C,
\qquad
\sup_\phi\|M_{\Delta t}(\phi)^{-1}-I\|\le C\Delta t.
\]
It follows that \(H_{\Delta t}\) is uniformly Lipschitz and
\[
H_{\Delta t}(z)=F(z)+O(\Delta t)
\]
uniformly in that neighborhood. Therefore,
\[
\|\Psi_{\Delta t}(z)-\Psi_{\Delta t}(\widehat z)\|
\le (1+C\Delta t)\|z-\widehat z\|,
\]
and Taylor expansion of the exact solution gives the local defect
\[
z(t_{n+1})
=
\Psi_{\Delta t}(z(t_n))+\rho^{n+1},
\qquad
\|\rho^{n+1}\|\le C\Delta t^2.
\]
With \(e^n=z^n-z(t_n)\), we obtain
\[
\|e^{n+1}\|
\le (1+C\Delta t)\|e^n\|+C\Delta t^2.
\]
Since \(e^0=0\), the discrete Gronwall inequality yields
\(\max_{n\le T/\Delta t}\|e^n\|\le C_T\Delta t\) as long as the numerical
trajectory remains in the chosen neighborhood. Choose a smaller compact
neighborhood whose boundary has positive distance \(d\) from the exact
trajectory, and apply the estimate up to the first possible exit time.
For \(\Delta t<d/(2C_T)\), the error bound precludes such an exit. Hence
the estimate holds on the full interval \([0,T]\).
\end{proof}

\begin{remark}
The same proof applies to any prescribed sequence of step-dependent
decompositions for which the positivity, boundedness, and Lipschitz
constants above are uniform in \(n\). An adaptive rule with discontinuous
switching requires a separate stability analysis of the selection map.
\end{remark}

\subsection{Step-dependent component decompositions}
\label{subsec:step-dependent-decompositions}

A further consequence of using a single auxiliary variable is that the component decomposition used in the correction need not be fixed. Since the scalar auxiliary variable tracks only the total quantity $Q(\phi)^2=E_1[\phi]+C$, the total shift \(C\) is kept fixed while the components may be recombined or repartitioned from one step to the next. In particular, if a step-dependent decomposition is obtained by grouping a fixed base decomposition, then for every block \(G\) we use
\[
E_{1,G}[\phi]:=\sum_{\alpha\in G}E_{1,\alpha}[\phi],
\qquad
C_G:=\sum_{\alpha\in G}C_\alpha.
\]
Thus \(\sum_G C_G=C\), independently of the grouping and of the time level. At step $n$, let
\begin{equation}
E_1[\phi]+C
=
\sum_{\alpha=1}^{r_n}
\left(E_{1,\alpha}^n[\phi]+C_\alpha^n\right),
\qquad
E_{1,\alpha}^n[\phi]+C_\alpha^n>0,
\label{eq:s3-step-split}
\end{equation}
and define
\begin{equation}
\mathcal B_{r_n}^n\psi
:=
\sum_{\alpha=1}^{r_n}
\frac{1}{2(Q_\alpha^n)^2}
\left(U_\alpha^n,\psi\right)U_\alpha^n,
\label{eq:s3-step-B}
\end{equation}
with
\[
Q_\alpha^n:=\sqrt{E_{1,\alpha}^n[\phi^n]+C_\alpha^n},
\qquad
U_\alpha^n:=\frac{\delta E_{1,\alpha}^n}{\delta\phi}[\phi^n].
\]
Replacing $\mathcal B_r^n$ in \eqref{eq:s3-rank-r-a} by
$\mathcal B_{r_n}^n$ and keeping the scalar update
\eqref{eq:s3-rank-r-b} gives
\begin{align}
\frac{\phi^{n+1}-\phi^n}{\Delta t}
&=
\mathcal G\left(
\mathcal L\phi^{n+1}
+
\frac{q^n}{Q^n}U^n
+
\mathcal B_{r_n}^n(\phi^{n+1}-\phi^n)
\right),
\label{eq:s3-step-rank-a}
\\
\frac{q^{n+1}-q^n}{\Delta t}
&=
\frac{1}{2Q^n}
\left(
U^n,\frac{\phi^{n+1}-\phi^n}{\Delta t}
\right).
\label{eq:s3-step-rank-b}
\end{align}

\begin{theorem}
\label{thm:s3-step-dependent-stability}
Assume that the decomposition
\eqref{eq:s3-step-split} is admissible for every step $n$. Let
\[
\mathcal S_{r_n}^n:=\mathcal B_{r_n}^n-\mathcal B_1^n.
\]
Then $\mathcal S_{r_n}^n\succeq0$, and the
step-dependent scheme \eqref{eq:s3-step-rank-a}--\eqref{eq:s3-step-rank-b}
satisfies
\begin{align}
&\frac{1}{\Delta t}
\left[
\widetilde E[\phi^{n+1},q^{n+1}]
-
\widetilde E[\phi^n,q^n]
\right]
+
\frac{(q^{n+1}-q^n)^2}{\Delta t}
+
\frac{1}{2\Delta t}\bigl(\phi^{n+1}-\phi^n,\mathcal L(\phi^{n+1}-\phi^n)\bigr)
\nonumber\\
&\qquad
+
\frac{1}{\Delta t}
\left(
\phi^{n+1}-\phi^n,
\mathcal S_{r_n}^n(\phi^{n+1}-\phi^n)
\right)
=
\left(
\Xi_{\mathrm{sd}}^{n+1},
\mathcal G\Xi_{\mathrm{sd}}^{n+1}
\right)
\le0,
\label{eq:s3-step-energy-law}
\end{align}
where
\[
\Xi_{\mathrm{sd}}^{n+1}
:=
\mathcal L\phi^{n+1}
+
\frac{q^n}{Q^n}U^n
+
\mathcal B_{r_n}^n(\phi^{n+1}-\phi^n).
\]
In particular, $\widetilde E[\phi^{n+1},q^{n+1}]
\le \widetilde E[\phi^n,q^n]$ for every admissible step-dependent decomposition
sequence.
\end{theorem}

\begin{proof}
Fix a time level $n$ and set
\(\delta\phi^{n+1}=\phi^{n+1}-\phi^n\).  The step-$n$ admissible decomposition implies
\[
\sum_{\alpha=1}^{r_n}U_\alpha^n=U^n,
\qquad
\sum_{\alpha=1}^{r_n}(Q_\alpha^n)^2=(Q^n)^2.
\]
Therefore Lemma~\ref{lem:B-ordering}, applied to this particular decomposition, gives
\(\mathcal S_{r_n}^n=\mathcal B_{r_n}^n-\mathcal B_1^n\succeq0\).  Moreover, the scalar update \eqref{eq:s3-step-rank-b} gives
\[
\mathcal B_1^n\delta\phi^{n+1}
=
\frac{q^{n+1}-q^n}{Q^n}U^n.
\]
Hence
\[
\Xi_{\mathrm{sd}}^{n+1}
=
\mathcal L\phi^{n+1}
+
\frac{q^{n+1}}{Q^n}U^n
+
\mathcal S_{r_n}^n\delta\phi^{n+1}.
\]
Taking the inner product of \eqref{eq:s3-step-rank-a} with
\(\Xi_{\mathrm{sd}}^{n+1}\) and using the two identities
\[
(\mathcal L\phi^{n+1},\delta\phi^{n+1})
=
\tfrac12(\phi^{n+1},\mathcal L\phi^{n+1})-\tfrac12(\phi^n,\mathcal L\phi^n)
+\tfrac12(\delta\phi^{n+1},\mathcal L\delta\phi^{n+1})
\]
and
\[
\frac{q^{n+1}}{Q^n}(U^n,\delta\phi^{n+1})
=(q^{n+1})^2-(q^n)^2+(q^{n+1}-q^n)^2
\]
gives \eqref{eq:s3-step-energy-law}.  The right-hand side is nonpositive because
$\mathcal G$ is nonpositive, and the additional term involving
$\mathcal S_{r_n}^n$ is nonnegative by the first part of the proof.
\end{proof}

\begin{remark}
\label{rem:step-dependent-vs-msav}
The step-dependent construction uses the fact that the auxiliary
variable tracks only the total quantity \(Q(\phi)^2=E_1[\phi]+C\). The component
decomposition at step \(n\) enters only through the pullback correction
\(\mathcal B_{r_n}^n\). Hence the stability proof requires only the stepwise
identities
\[
\sum_{\alpha=1}^{r_n} U_\alpha^n = U^n,
\qquad
\sum_{\alpha=1}^{r_n} (Q_\alpha^n)^2 = (Q^n)^2,
\]
and does not require component auxiliary variables to be matched across time
steps. This differs from MSAV, where the modified energy contains the
componentwise variables \(q_\alpha\).
\end{remark}

\subsection{An adaptive grouping algorithm}
\label{subsec:variance-guided-grouping}

The step-dependent formulation lets the grouping of the base decomposition change
from one step to the next, and the weighted-variance identity provides a selection criterion. Recall the base decomposition
$E_1[\phi]+C=\sum_{\alpha=1}^{r}\bigl(E_{1,\alpha}[\phi]+C_\alpha\bigr)$, with
correction $\mathcal B_r^n$. A grouping $\mathcal P$ partitions the components $\{1,\dots,r\}$ into blocks $G$, with the correction $\mathcal B_{\mathcal P}^n$; the single block recovers $\mathcal B_1^n$, and resolving every component separately recovers $\mathcal B_r^n$. Fix a direction
$\psi^n$ (for example $\phi^n-\phi^{n-1}$, or a rank-one SAV predictor). With
$w_\alpha=(Q_\alpha^n)^2$ and, as in Lemma~\ref{lem:B-ordering} and
Proposition~\ref{prop:refinement-monotonicity},
\[
y_\alpha=\frac{(U_\alpha^n,\psi^n)}{w_\alpha},
\qquad
\bar y_G=\frac{\sum_{\alpha\in G}(U_\alpha^n,\psi^n)}{\sum_{\alpha\in G}w_\alpha},
\]
the part of the rank-$r$ gap that $\mathcal P$ gives up in the direction $\psi^n$ is
\[
\mathcal V(\mathcal P)
:=\bigl(\psi^n,(\mathcal B_r^n-\mathcal B_{\mathcal P}^n)\psi^n\bigr)
=\frac12\sum_{G\in\mathcal P}\sum_{\alpha\in G}w_\alpha\,(y_\alpha-\bar y_G)^2 ,
\]
which, by Proposition~\ref{prop:refinement-monotonicity}, does not increase under
refinement: it is the full gap $(\psi^n,(\mathcal B_r^n-\mathcal B_1^n)\psi^n)$ for
the single block and vanishes when every component is resolved separately.

Algorithm~\ref{alg:adaptive-grouping} uses a restricted class of these groupings.
Let $M_0\subseteq\{1,\dots,r\}$ be the components resolved separately and
$M_1=\{1,\dots,r\}\setminus M_0$ the rest, merged into one block.
The resulting correction has \( |M_0|+1 \) grouped directions, or \(r\) when
\(M_1=\varnothing\); its actual matrix rank is at most this nominal rank.
Writing $\mathcal V(M_0)$ for
$\mathcal V$ of this grouping, we have:
\[
\mathcal V(M_0)
=
\frac12\sum_{\alpha\in M_1}w_\alpha\,(y_\alpha-\bar y_{M_1})^2,
\qquad
\bar y_{M_1}=\frac{\sum_{\alpha\in M_1}(U_\alpha^n,\psi^n)}{\sum_{\alpha\in M_1}w_\alpha}.
\]
The full gap is $\mathcal V(\varnothing)$, and $\mathcal V(M_0)$ decreases to $0$ as
$M_0$ grows to all of $\{1,\dots,r\}$. We refine until the recovered fraction
reaches a target $\tau\in(0,1]$, i.e.\ until
$\mathcal V(M_0)\le(1-\tau)\,\mathcal V(\varnothing)$, separating at each step,
greedily, the component that reduces $\mathcal V$ the most.

\begin{algorithm}[ht]
\caption{Greedy adaptive grouping for the step-$n$ PB-SAV update}
\label{alg:adaptive-grouping}
\begin{algorithmic}[1]
\Require base decomposition $\{E_{1,\alpha}\}_{\alpha=1}^{r}$; direction $\psi^n$;
retention fraction $\tau\in(0,1]$
\State $M_0\gets\varnothing$ \Comment{all components merged; correction $\mathcal B_1^n$}
\While{$\mathcal V(M_0)>(1-\tau)\,\mathcal V(\varnothing)$}
  \State $\displaystyle\alpha^\star\in\operatorname*{arg\,min}_{\alpha\notin M_0}
  \mathcal V(M_0\cup\{\alpha\})$ \Comment{greedy step}
  \State $M_0\gets M_0\cup\{\alpha^\star\}$
\EndWhile
\State form the grouping from $M_0$ and take the PB-SAV update
\eqref{eq:s3-step-rank-a}--\eqref{eq:s3-step-rank-b}
\end{algorithmic}
\end{algorithm}

Because Algorithm~\ref{alg:adaptive-grouping} returns an admissible grouping at the
current step, Theorem~\ref{thm:s3-step-dependent-stability} applies to any sequence
it produces; the variance criterion is an a posteriori rule for selecting which
component directions are active in the current update.

\section{Finite-dimensional optimization viewpoint}
\label{sec:optimization-viewpoint}

We now give a finite-dimensional interpretation of the correction. It relates PB-SAV to Gauss--Newton and motivates the relaxation of Section~\ref{subsec:q2-relaxation}.

Consider the unconstrained problem
\[
\min_{\phi\in\mathbb R^m} f(\phi),
\qquad
f(\phi)=E_1(\phi).
\]
In Sections~\ref{subsec:s4-gradient-step} and~\ref{subsec:s4-convex-quadratic} we take $\mathcal L=0$ and $\mathcal G=-I_m$, and write $Q(\phi)=\sqrt{f(\phi)+C}$; the relaxation of Section~\ref{subsec:q2-relaxation} is stated for general $\mathcal G$ and $\mathcal L$.

\subsection{SAV and PB-SAV as gradient steps}
\label{subsec:s4-gradient-step}

For SAV, substituting the scalar update into the state equation gives
\begin{equation}
\phi^{n+1}
=
\phi^n-\alpha_n P_1^n \nabla f(\phi^n),
\qquad
\alpha_n:=\Delta t\,\frac{q^n}{Q(\phi^n)},
\qquad
P_1^n:=\bigl(I+\Delta t\,\mathcal B_1^n\bigr)^{-1},
\label{eq:s4-plain-iterate}
\end{equation}
where $\mathcal B_1^n$ is rank one,
\begin{equation}
\mathcal B_1^n=\frac{\nabla f(\phi^n)\nabla f(\phi^n)^T}{2Q(\phi^n)^2}.
\label{eq:s4-plain-rank-one-correction}
\end{equation}
Since $\mathcal B_1^n\succeq0$, the matrix
\(I+\Delta t\,\mathcal B_1^n\) is positive definite, so \(P_1^n\) is
well defined. The positivity of the effective step length follows from the
next observation.

\begin{proposition}
\label{prop:q-positivity}
In the finite-dimensional optimization setting
\(\mathcal L=0\), \(\mathcal G=-I_m\), assume
\(q^0=Q(\phi^0)>0\). For either SAV or PB-SAV, let
\(\mathcal S_r^n=\mathcal B_r^n-\mathcal B_1^n\), with
\(\mathcal S_r^n=0\) in the SAV case. Then
\begin{equation}
q^{n+1}
=
\frac{q^n}{
1+\dfrac{\Delta t}{2(Q^n)^2}
\left(
U^n,(I+\Delta t\,\mathcal S_r^n)^{-1}U^n
\right)}
>0.
\label{eq:s4-q-positive}
\end{equation}
Consequently, \(q^n>0\) for every \(n\).
\end{proposition}

\begin{proof}
Using
\(\mathcal B_r^n=\mathcal B_1^n+\mathcal S_r^n\) and the scalar update,
the PB-SAV state equation can be written as
\[
(I+\Delta t\,\mathcal S_r^n)(\phi^{n+1}-\phi^n)
=
-\Delta t\,\frac{q^{n+1}}{Q^n}U^n.
\]
Substituting this identity into the scalar update gives
\eqref{eq:s4-q-positive}. Since
\(\mathcal S_r^n\succeq0\), the denominator is at least one, and the
conclusion follows by induction. The SAV case is obtained by taking
\(\mathcal S_r^n=0\).
\end{proof}

It follows that \(\alpha_n>0\) in \eqref{eq:s4-plain-iterate}; hence
\(-P_1^n\nabla f(\phi^n)\) is a preconditioned descent direction. Likewise,
the PB-SAV increment is along the preconditioned direction
\(-(I+\Delta t\,\mathcal S_r^n)^{-1}\nabla f(\phi^n)\).

For decomposition \(f(\phi)=\sum_{\alpha=1}^r f_\alpha(\phi)\), define
\[
Q_\alpha(\phi)^2=f_\alpha(\phi)+C_\alpha,
\qquad
R(\phi):=(Q_1(\phi),\ldots,Q_r(\phi)).
\]
Then \(f(\phi)+C=\|R(\phi)\|_2^2\) with $C=\sum_{\alpha=1}^r C_\alpha$, and in finite dimensions
\[
\mathcal B_r^n
=
\sum_{\alpha=1}^r
\frac{\nabla f_\alpha(\phi^n)\nabla f_\alpha(\phi^n)^T}
{2Q_\alpha(\phi^n)^2}=
2\,DR(\phi^n)^TDR(\phi^n).
\]
Thus PB-SAV uses the Gauss--Newton part of the least-squares energy
\(\|R(\phi)\|_2^2\), while standard SAV uses the corresponding scalar pullback
for \(Q(\phi)=\|R(\phi)\|_2\).

\subsection{Model problem: convex quadratic}
\label{subsec:s4-convex-quadratic}

The effect of the pullback correction is easiest to see for a diagonal convex quadratic,
\begin{equation}
f(\phi)=\frac{1}{2}\phi^TA\phi
=\frac{1}{2}\sum_{i=1}^m a_i \phi_i^2,
\qquad
a_i>0.
\label{eq:s4-quadratic}
\end{equation}
With the coordinate decomposition $f_i(\phi)=\tfrac12 a_i \phi_i^2$ and $C_i>0$, each component has $\nabla f_i(\phi)=a_i\phi_i\,e_i$ (only its $i$-th coordinate is nonzero) and $Q_i(\phi)^2=f_i(\phi)+C_i=\tfrac12 a_i\phi_i^2+C_i$. The contribution from component $i$ to $\mathcal B_r^n$ is the rank-one matrix
\[
\frac{\nabla f_i(\phi^n)\nabla f_i(\phi^n)^T}{2Q_i(\phi^n)^2}
=
\frac{a_i^2(\phi_i^n)^2}{a_i(\phi_i^n)^2+2C_i}\,e_i e_i^T,
\]
so summing over $i=1,\dots,m$ gives a diagonal correction:
\begin{equation}
\mathcal B_m^n
=
\operatorname{diag}(b_1^n,\dots,b_m^n),
\qquad
b_i^n:=
\frac{a_i^2(\phi_i^n)^2}{a_i(\phi_i^n)^2+2C_i}
\in[0,a_i].
\label{eq:s4-quadratic-Br}
\end{equation}
When $C_i$ is small relative to $f_i(\phi^n)$, $b_i^n\approx a_i$, so $\mathcal B_m^n$ approximates the Hessian $A$. In that regime
\[
P_m^n
=
\bigl(I+\Delta t\,\mathcal B_m^n\bigr)^{-1}
\approx
\bigl(I+\Delta t\,A\bigr)^{-1},
\]
matching the implicit Euler step for the quadratic model. By contrast, SAV uses only
\begin{equation}
\mathcal B_1^n
=
\frac{(A\phi^n)(A\phi^n)^T}{(\phi^n)^TA\phi^n+2C},
\label{eq:s4-quadratic-B1}
\end{equation}
which is rank one. Grouped decompositions lie between these two cases: $\mathcal B_r^n$ is built from grouped component gradients rather than from the total gradient alone.

\subsection{Relaxation of the auxiliary variable}
\label{subsec:q2-relaxation}

SAV-type schemes have also been used as optimization methods \cite{liu2023efficient,zhang2025relaxed}.  In that setting the time step is often chosen much larger than it would be in a time-accurate computation.  The modified energy may still decrease, but the numerical auxiliary variable $q^n$ can drift away from its defining value $Q(\phi^n)$.  If the ratio $q^n/Q(\phi^n)$ becomes small, the scalar multiplier in the state update suppresses the explicit nonlinear force and the iteration may stagnate before reaching a minimizer.  PB-SAV uses the same scalar auxiliary variable, so this drift of $q^n$ from $Q(\phi^n)$ can arise even when the pullback correction supplies a local metric.

We therefore use a relaxation of the auxiliary variable that preserves the modified-energy law. Let the unrelaxed update \eqref{eq:s3-rank-r-a}--\eqref{eq:s3-rank-r-b} produce $(\phi^{n+1},\widetilde q^{n+1})$:
\begin{align}
\frac{\phi^{n+1}-\phi^n}{\Delta t}
&=
\mathcal G\left(
\mathcal L\phi^{n+1}
+
\frac{q^n}{Q^n}U^n
+
\mathcal B_r^n(\phi^{n+1}-\phi^n)
\right),
\label{eq:s4-unrelaxed-rank-r-a}
\\
\frac{\widetilde q^{n+1}-q^n}{\Delta t}
&=
\frac{1}{2Q^n}
\left(
U^n,\frac{\phi^{n+1}-\phi^n}{\Delta t}
\right).
\label{eq:s4-unrelaxed-rank-r-b}
\end{align}
The unrelaxed update satisfies $\widetilde E[\phi^{n+1},\widetilde q^{n+1}]\le\widetilde E[\phi^n,q^n]$, and we define the dissipation margin
\[
D_n
:=
\widetilde E[\phi^n,q^n]-\widetilde E[\phi^{n+1},\widetilde q^{n+1}]\ge0.
\]
Given an admissible relaxation amount $\delta_n\in[0,D_n]$, the relaxation sets
\begin{equation}
q^{n+1}
:=
\sqrt{\min\bigl(Q(\phi^{n+1})^2,\,(\widetilde q^{n+1})^2+\delta_n\bigr)}.
\label{eq:s4-energy-admissible-relaxation}
\end{equation}
The relaxed value is chosen so that the modified-energy dissipation law is preserved. We take the positive root so that $q^{n+1}\ge0$; the modified energy depends only on $q^2$.

\begin{proposition}
\label{lem:q2-relaxation-energy}
Under the assumptions above, the relaxed pair satisfies
\begin{equation}
\widetilde E[\phi^{n+1},q^{n+1}]
\le
\widetilde E[\phi^n,q^n]-(D_n-\delta_n)
\le
\widetilde E[\phi^n,q^n].
\label{eq:s4-energy-admissible-relaxation-law}
\end{equation}
\end{proposition}

\begin{proof}
$\phi^{n+1}$ is unchanged, so
\[
\widetilde E[\phi^{n+1},q^{n+1}]
-
\widetilde E[\phi^{n+1},\widetilde q^{n+1}]
=
(q^{n+1})^2-(\widetilde q^{n+1})^2
\le
\delta_n
\]
by \eqref{eq:s4-energy-admissible-relaxation}, hence
$\widetilde E[\phi^{n+1},q^{n+1}]
\le \widetilde E[\phi^n,q^n]-(D_n-\delta_n)$.
\end{proof}

The squared defect from \(Q(\phi^{n+1})\) is
\[
Q(\phi^{n+1})^2-(q^{n+1})^2
=
\bigl[Q(\phi^{n+1})^2-(\widetilde q^{n+1})^2-\delta_n\bigr]_+.
\]

Every choice $\delta_n\in[0,D_n]$ satisfies Proposition~\ref{lem:q2-relaxation-energy}.  The unrelaxed energy law of Theorem~\ref{thm:rank-r-stability}, applied to $(\phi^{n+1},\widetilde q^{n+1})$, decomposes $D_n$ into nonnegative contributions, including
\[
S_n:=\bigl(\phi^{n+1}-\phi^n,\,\mathcal S_r^n(\phi^{n+1}-\phi^n)\bigr);
\]
in particular $S_n\in[0,D_n]$. Three concrete choices are
\begin{equation}
\delta_n=\rho D_n,
\qquad
\delta_n=\rho S_n,
\qquad
\delta_n=S_n+\rho(D_n-S_n),
\label{eq:s4-delta-choices}
\end{equation}
with $\rho \in[0,1]$. In the finite-dimensional tests below we use the maximal choice $\delta_n=D_n$.

\section{Numerical experiments}
\label{sec:numerics}

We now report numerical experiments. PDE reference solutions are computed by ETDRK4 \cite{cox2002exponential} with sufficiently small reference time steps, and the spatial discretization is Fourier spectral on the periodic domain $\Omega=[0,2\pi]^2$. Unless stated otherwise, a shift \(C_\alpha=10^{-12}\) is assigned to each component of a fixed finest decomposition. Under grouping, the block shift is inherited as \(C_G=\sum_{\alpha\in G}C_\alpha\); hence the total scalar shift \(C=\sum_\alpha C_\alpha\) is independent of the correction rank and of the time level. These shifts keep every \(Q_\alpha^2=E_{1,\alpha}+C_\alpha\) strictly positive, are numerically negligible, and do not change the component forces or the reference equations. In all reported SAV, PB-SAV, relaxed, and MSAV runs, the corresponding modified energy is nonincreasing to roundoff. The figures therefore emphasize trajectory accuracy, the auxiliary-variable defect $|q^n-Q(\phi^n)|$, and errors against the reference solution.

\subsection{Finite-dimensional test problems}
\label{subsec:exp-optimization}

We first test the finite-dimensional formulation of Section~\ref{sec:optimization-viewpoint} on two problems where the effect of the correction can be read directly from the component forces.  With $\mathcal G=-I_m$, the gradient flow is $\partial_t\phi=-\nabla f(\phi)$. 

\paragraph{Convex quadratic with two distinct eigenvalues}
We begin with $\mathcal L=0$ and the test problem
\begin{equation}
f(\phi)
=
\sum_{k=1}^{50}\phi_{2k-1}^2
+ \frac{1}{100}\sum_{k=1}^{50}\phi_{2k}^2,
\qquad
\phi^0=(1,\ldots,1),
\label{eq:s6-opt-quadratic}
\end{equation}
with component decomposition $f_i(\phi)=\lambda_i \phi_i^2$, where $\lambda_i=1$ on odd coordinates and $\lambda_i=10^{-2}$ on even coordinates.  This objective function has exactly two distinct eigenvalues, and the rank-$2$ decomposition used below groups components by curvature class. We denote by $f^*$ its minimum value.

At $\Delta t=10$, Figure~\ref{fig:optimization-relax-convex} compares SAV and rank-$2$ PB-SAV with and without relaxation. Without relaxation, the rank-$2$ PB-SAV update reduces the objective faster than SAV in the early iterations and lowers the eventual plateau by roughly a factor of five: by iteration $300$ the objective gaps are $4.95\times10^{-1}$ for SAV and $9.31\times10^{-2}$ for rank-$2$ PB-SAV. The large step size, however, drives $q^n$ far below $Q(\phi^n)$ in the first few steps; by iteration $5$ the ratio $q^n/Q(\phi^n)$ has fallen to $2.46\times10^{-3}$ for SAV and $2.99\times10^{-1}$ for rank-$2$ PB-SAV, and by iteration $300$ to roundoff level. The correction determines the local metric, whereas the ratio $q^n/Q(\phi^n)$ scales the explicit nonlinear force.

To restore consistency between $q^n$ and $Q(\phi^n)$, we apply the relaxation~\eqref{eq:s4-energy-admissible-relaxation}
\[
q^{n+1}
:=
\sqrt{\min\bigl(Q(\phi^{n+1})^2,\,(\widetilde q^{n+1})^2+\delta_n\bigr)},
\]
with the maximal admissible step $\delta_n=D_n$, and denote the resulting rank-$1$ and rank-$2$ schemes by relaxed SAV and rank-$2$ relaxed PB-SAV. Figure~\ref{fig:optimization-relax-convex}(b) shows that the relaxation keeps $q^n/Q(\phi^n)$ above $0.49$ for relaxed SAV and keeps it equal to $1$ for the first $65$ iterations of rank-$2$ relaxed PB-SAV. Panel (a) shows the objective gap. Rank-$2$ relaxed PB-SAV crosses $10^{-10}$ at iteration $63$, while relaxed SAV converges to the same optimum more slowly, reaching $7.69\times10^{-10}$ only at iteration $300$. Thus the relaxation keeps the scalar factor large enough for both schemes to continue descending toward the minimizer, and the rank-$2$ correction reduces the iteration count by roughly an order of magnitude in this test.

\begin{figure}[htbp]
  \centering
  \includegraphics[width=0.86\textwidth]{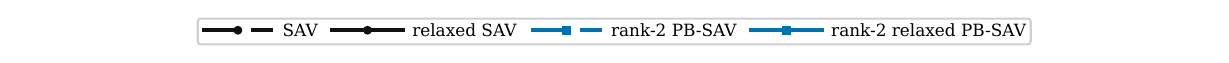}
  \vspace{-0.35em}

  \begin{subfigure}[t]{0.48\textwidth}
    \centering
    \includegraphics[width=\textwidth]{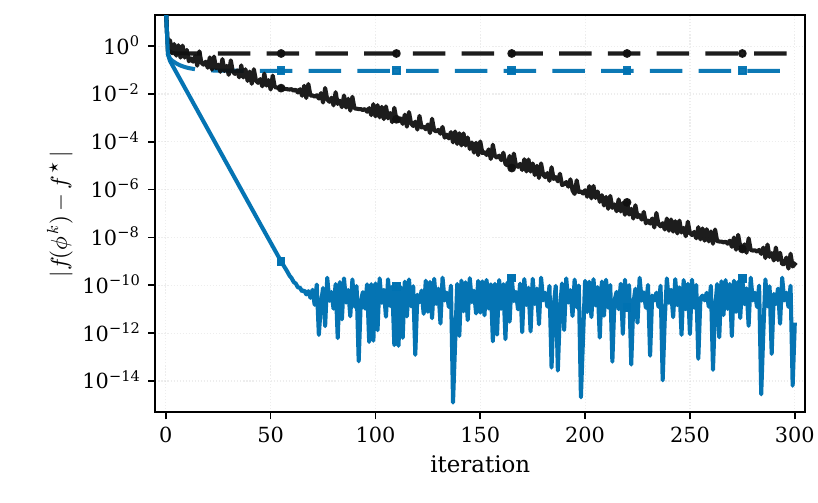}
    \caption{Objective gaps.}
  \end{subfigure}
  \hfill
  \begin{subfigure}[t]{0.48\textwidth}
    \centering
    \includegraphics[width=\textwidth]{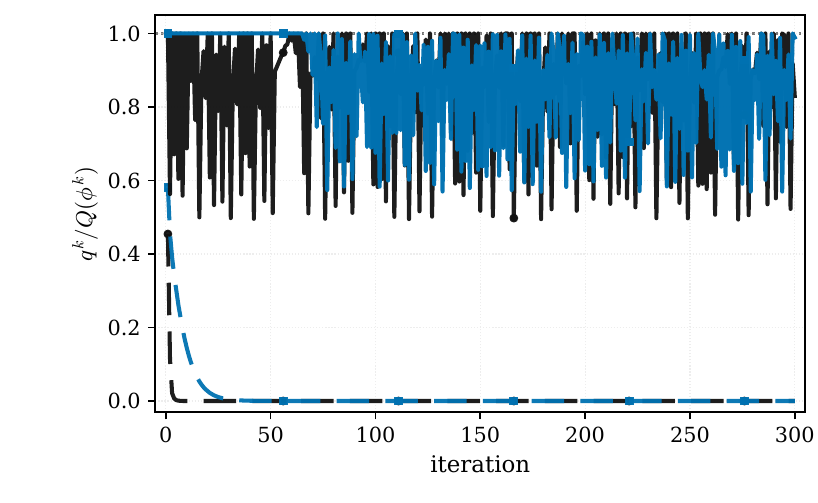}
    \caption{Auxiliary-variable ratio.}
  \end{subfigure}
  \caption{Convex quadratic \eqref{eq:s6-opt-quadratic} at $\Delta t=10$, comparing SAV and rank-$2$ PB-SAV with and without the $\delta_n=D_n$ relaxation. Panel (a) gives objective gaps; panel (b) gives $q^n/Q(\phi^n)$.}
  \label{fig:optimization-relax-convex}
\end{figure}

\paragraph{Regularized least squares with a quadratic part}
The second finite-dimensional test includes a nonzero quadratic part, as in the PDE splittings below.  All vectors are in $\mathbb R^{80}$ and $\Delta t=5$.  We consider the regularized least-squares objective
\[
f(\phi)
=
\tfrac12\phi^T L\phi
+\tfrac12(a_1^T\phi-b_1)^2
+\tfrac12(a_2^T\phi-b_2)^2,
\]
with quadratic part $\tfrac12\phi^T L\phi$; the nonlinear part $E_1$ is the sum of the two residual components. We take
\[
L=\operatorname{diag}(\ell_1,\ldots,\ell_{80}),
\qquad
\ell_j\text{ logarithmically spaced from }0.03\text{ to }0.3.
\]
Let
\(v_1,v_2\in\mathbb R^{80}\) be orthonormal and set
\[
a_1=4v_1,\qquad
a_2=2.8\bigl(\cos\theta\,v_1+\sin\theta\,v_2\bigr),
\qquad
\cos\theta=0.35 .
\]
Thus \(\|a_1\|=4\), \(\|a_2\|=2.8\), and
\(a_1^Ta_2/( \|a_1\|\|a_2\|)=0.35\). We choose a unit vector
\(\phi^\star\) and set \(b_i=a_i^T\phi^\star\), so that \(\phi^\star\) is the
unregularized least-squares target. The initial point is
\(\phi^0=2.5\,\widehat\phi^0\) with \(\|\widehat\phi^0\|=1\).

Here
\[
\nabla^2E_1=a_1a_1^T+a_2a_2^T
\]
is a constant rank-\(2\) matrix. At $\phi^0$, $\mathcal B_2^0$ approximates $\nabla^2 E_1$ to relative Frobenius error $1.92\times10^{-12}$, while SAV's rank-$1$ correction $\mathcal B_1^0$ has relative Frobenius error $3.42\times10^{-1}$. Figure~\ref{fig:optimization-lowrank-ls} shows the two relaxed methods (relaxed SAV and rank-$2$ relaxed PB-SAV): at iteration $200$, the objective gap is $4.82\times10^{-3}$ for relaxed SAV and below \(10^{-18}\) for rank-$2$ relaxed PB-SAV. The rank-$2$ correction recovers the Hessian $\nabla^2E_1$ to roundoff at $\phi^0$, and the resulting objective convergence is faster than relaxed SAV by orders of magnitude. 

\begin{figure}[htbp]
  \centering
  \begin{subfigure}[t]{0.48\textwidth}
    \centering
    \includegraphics[width=\textwidth]{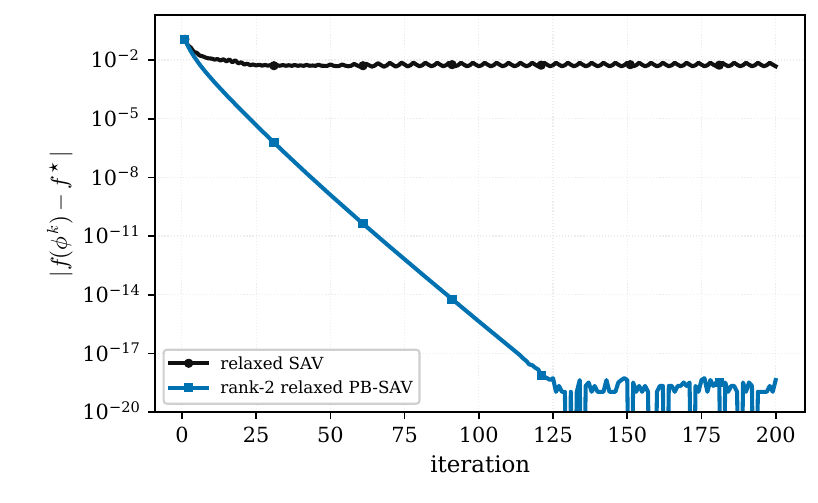}
    \caption{Objective gaps.}
  \end{subfigure}
  \hfill
  \begin{subfigure}[t]{0.48\textwidth}
    \centering
    \includegraphics[width=\textwidth]{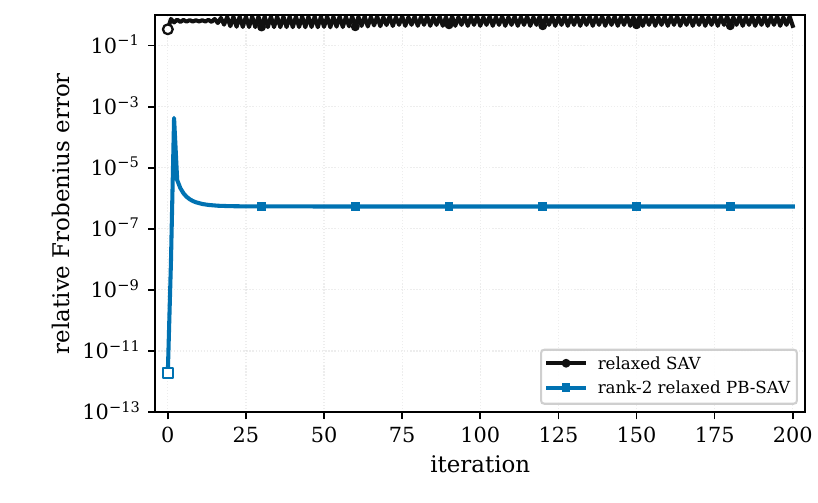}
    \caption{Correction-matrix error.}
  \end{subfigure}
  \caption{Regularized least squares with a quadratic part, showing the two relaxed schemes. Panel (a) gives objective gaps; panel (b) gives the relative Frobenius error of $\mathcal B_r^n$ as an approximation to the Hessian $\nabla^2 E_1$.}
  \label{fig:optimization-lowrank-ls}
\end{figure}

\subsection{Allen--Cahn equation}
\label{subsec:exp-ac}

We next consider the Allen--Cahn energy
\begin{equation}
E[\phi]=\frac12\int_\Omega |\nabla\phi|^2\,dx+E_1[\phi],
\qquad
E_1[\phi]=\int_\Omega \frac{(\phi^2-1)^2}{4\varepsilon^2}\,dx,
\label{eq:s6-ac}
\end{equation}
with $L^2$-gradient flow
\[
\partial_t\phi
=
-\frac{\delta E}{\delta\phi}
=
\Delta\phi-\frac{\phi^3-\phi}{\varepsilon^2}.
\]
The two experiments below use different admissible decompositions of the same nonlinear energy $E_1$.

\subsubsection{Convergence on a smooth initial condition}
\label{subsubsec:exp-ac-convergence}

The first Allen--Cahn test verifies first-order temporal convergence for SAV, PB-SAV, and MSAV, and also records the auxiliary-variable error. Here $\varepsilon^2=1.6\times10^{-3}$. Let \(\eta=10^{-12}\), set
\[
C=\frac{3|\Omega|}{4\varepsilon^2}+2\eta,
\qquad C_1=C_2=\eta,
\]
and use the two-component decomposition
\begin{equation}
E_{1,1}[\phi]=\int_\Omega \frac{(\phi^2-2)^2}{4\varepsilon^2}\,dx,
\qquad
E_{1,2}[\phi]=\int_\Omega \frac{\phi^2}{2\varepsilon^2}\,dx.
\label{eq:s6-ac-split}
\end{equation}
Then \(E_{1,\alpha}[\phi]+C_\alpha>0\) for \(\alpha=1,2\), so the decomposition is strictly admissible. Its variational derivatives satisfy
\[
\frac{\delta E_{1,1}}{\delta\phi}+\frac{\delta E_{1,2}}{\delta\phi}
=\frac{\phi^3-\phi}{\varepsilon^2},
\]
and
\[
\bigl(E_{1,1}[\phi]+C_1\bigr)
+\bigl(E_{1,2}[\phi]+C_2\bigr)
=
E_1[\phi]+C.
\]
Thus the single SAV auxiliary variable is defined from \(E_1+C\), while rank-$2$ PB-SAV and MSAV use the two strictly positive shifted components. The added shifts change neither the Allen--Cahn force nor the reference PDE.

The initial condition is $\phi^0(x,y) = 0.6\sin x\sin y$; ETDRK4 with $\Delta t_{\mathrm{ref}}=2.5\times10^{-8}$ gives the reference solution at $T=10^{-3}$, and the step-size ladder is $\Delta t\in\{10^{-5},\,5\times 10^{-6},\,2.5\times 10^{-6},\,1.25\times 10^{-6}\}$.

Figure~\ref{fig:ac-convergence}(a) shows the final $L^2$ error $\|\phi^N-\phi_{\mathrm{ref}}(T)\|_{L^2}$ at $T=10^{-3}$ for the three schemes. All three follow the $\mathcal O(\Delta t)$ reference slope, with SAV attaining the smallest constant and rank-$2$ PB-SAV and MSAV nearly coincident at this resolution. Figure~\ref{fig:ac-convergence}(b) plots the maximum auxiliary-variable error
\[
\max_{0\le n\le T/\Delta t}|q^n-Q(\phi^n)|
\]
on the same initial condition. Each curve follows the first-order reference slope; SAV and rank-$2$ PB-SAV are numerically coincident, and MSAV attains the same consistency order with its componentwise auxiliary variables. The two-component split~\eqref{eq:s6-ac-split} therefore does not reduce the trajectory error; rank-$2$ PB-SAV coincides with MSAV at a larger error constant than SAV, while the three schemes have the same first-order auxiliary-variable consistency.

\begin{figure}[htbp]
  \centering
  \begin{subfigure}[t]{0.48\textwidth}
    \centering
    \includegraphics[width=\textwidth]{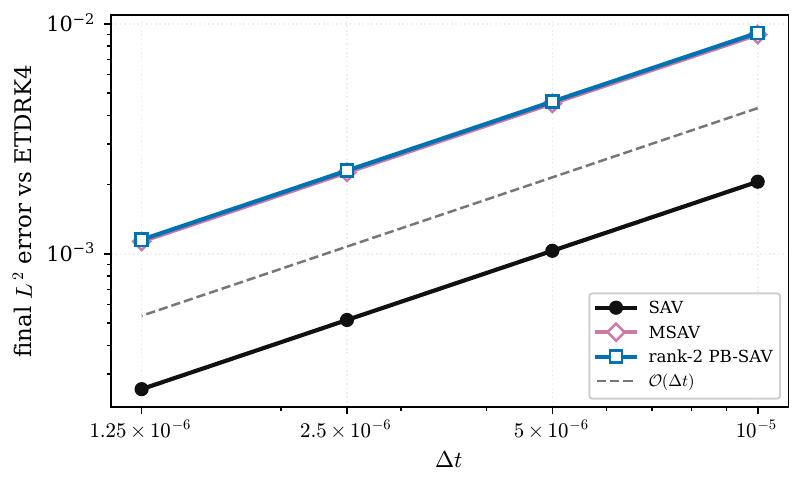}
    \caption{Final $L^2$ error vs $\Delta t$.}
  \end{subfigure}
  \hfill
  \begin{subfigure}[t]{0.48\textwidth}
    \centering
    \includegraphics[width=\textwidth]{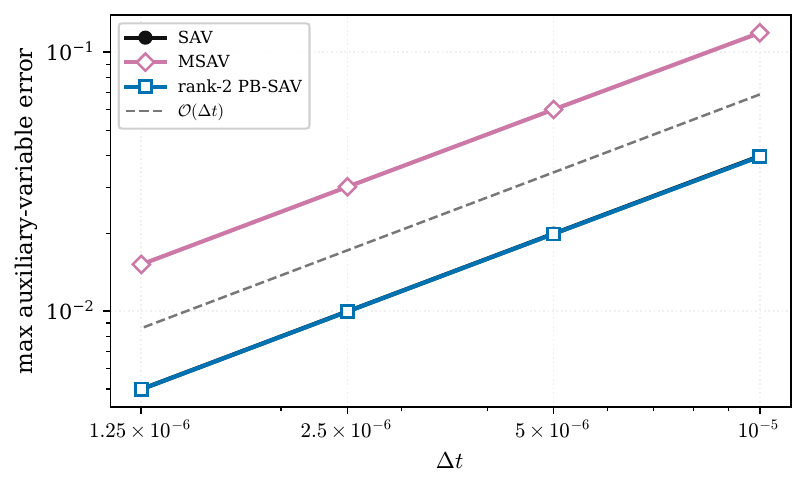}
    \caption{Maximum auxiliary-variable error vs $\Delta t$.}
  \end{subfigure}
  \caption{Convergence on the smooth Allen--Cahn initial condition $\phi^0=0.6\sin x\sin y$ at $T=10^{-3}$ for SAV, rank-$2$ PB-SAV, and MSAV, using \eqref{eq:s6-ac-split}. Panel (a) gives final $L^2$ errors; panel (b) gives $\max_n|q^n-Q(\phi^n)|$.}
  \label{fig:ac-convergence}
\end{figure}

\subsubsection{A Fourier-band residual split}
\label{subsubsec:exp-ac-residual-spectral}

The preceding split does not improve the trajectory error for the smooth initial condition.  We therefore use a different admissible split of the same Allen--Cahn energy, for which the correction reduces the error at the step sizes considered.

Write the double-well residual $\omega:=(\phi^2-1)/2$, so that $E_1[\phi]=\varepsilon^{-2}\int_\Omega \omega^2\,dx$. Let $\{\Pi_\alpha\}_{\alpha=1}^r$ be the orthogonal projections onto a partition of the Fourier grid into $r$ shells, chosen so that the initial residual energies $\varepsilon^{-2}\int_\Omega(\Pi_\alpha \omega)^2\,dx$ are approximately balanced across shells, and set
\[
\omega_\alpha:=\Pi_\alpha \omega,
\qquad
E_{1,\alpha}[\phi]:=\frac{1}{\varepsilon^2}\int_\Omega \omega_\alpha^2\,dx,
\qquad
U_\alpha=\frac{2}{\varepsilon^2}\,\phi\,\omega_\alpha.
\]
Since the $\Pi_\alpha$ are orthogonal and idempotent, $\sum_\alpha \omega_\alpha=\omega$, $\sum_\alpha E_{1,\alpha}=E_1$, and $\sum_\alpha U_\alpha=\varepsilon^{-2}\phi(\phi^2-1)$. 

We take the broadband initial condition $\phi^0=0.6\,(f-\bar f)/\|f-\bar f\|_{L^\infty}$ with
\[
\begin{aligned}
f(x,y)={}&0.45\sin x\sin y+0.25\cos(2x-y)+0.18\sin(4x+3y)\\
&+0.12\cos(7x-2y)+0.08\sin(10x+5y)
\end{aligned}
\]
on $[0,2\pi]^2$ (where $\bar f$ is the discrete mean), integrate to $T=6\times10^{-3}$ on a $128^2$ Fourier grid with $\varepsilon^2=9\times10^{-4}$, and measure the final $L^2$ error against an ETDRK4 reference at $\Delta t_{\mathrm{ref}}=2.5\times10^{-6}$.

Figure~\ref{fig:ac-residual-spectral} shows the phase fields and the error fields against the reference at the test step $\Delta t=7.5\times10^{-4}$. The SAV and rank-$2$ PB-SAV phase fields differ slightly from the reference, while the rank-$2$ error field is smaller than the SAV error field along the moving interfaces.

Quantitatively, the rank-$2$ correction reduces the final  $L^2$ error from $3.88\times10^{-2}$ for SAV to $2.55\times10^{-2}$, a ratio of $0.656$ (about a $34\%$ reduction); at $\Delta t=10^{-3}$ the ratio is $0.668$. Increasing the rank from $2$ to $4$ leaves the error unchanged at this precision, so rank $2$ accounts for the entire observed reduction in this example.

\begin{figure}[tbp]
  \centering
  \includegraphics[width=0.92\textwidth]{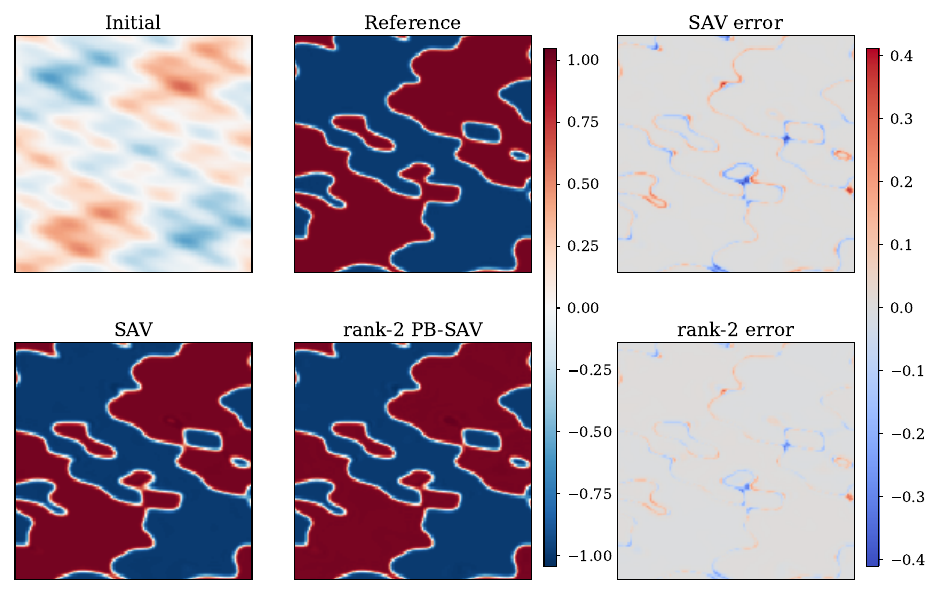}
  \caption{Fourier-band residual split for Allen--Cahn, broadband initial condition ($\varepsilon^2=9\times10^{-4}$, $T=6\times10^{-3}$, $\Delta t=7.5\times10^{-4}$, $128^2$ grid). The left four panels are the phase fields: the initial condition, the ETDRK4 reference, and the SAV and rank-$2$ PB-SAV final states. The right two panels are the error fields $\phi^N-\phi_{\mathrm{ref}}(T)$ for SAV and rank-$2$ PB-SAV.}
  \label{fig:ac-residual-spectral}
\end{figure}

\subsection{A finite-rank nonlocal model problem}
\label{subsec:exp-finite-rank-pde}

We next turn to a constructed finite-rank nonlocal model for which the component forces are prescribed.  We consider
\begin{equation}
\partial_t\phi
=
-\varepsilon^2(-\Delta)\phi
-\frac{\delta E_1}{\delta \phi},
\qquad
E[\phi]
=
\frac{\varepsilon^2}{2}\langle \phi,(-\Delta)\phi\rangle
+E_1[\phi],
\label{eq:s6-finite-rank-pde}
\end{equation}
with $\varepsilon^2=10^{-2}$, and
\begin{equation}
E_1[\phi]
=
\sum_{j=1}^{10} E_{1,j}[\phi],
\qquad
E_{1,j}[\phi]
=
\frac{\gamma_j}{4}
\left(\langle \phi,\psi_j\rangle^2-a_j^2\right)^2.
\label{eq:s6-finite-rank-energy}
\end{equation}

Here $\psi_j$ are normalized Fourier modes. In the displayed case, 
\[
\begin{aligned}
(\psi_j)_{j=1}^{10}
=&
\bigl(\cos x,\ \cos y,\ \cos(2x+y),\ \sin(2x-y),\ \cos(3x+2y),\\
&\quad
\sin(3x-y),\ \cos(4x+y),\ \sin(x+3y),\ \cos(5x+2y),\ \sin(2x+4y)\bigr),
\end{aligned}
\]
with normalization in the discrete $L^2$ inner product. We take
\(\gamma_j=9-\frac12(j-1)\) and \(a_j=1.25-0.05(j-1)\). The displayed initial condition is the balanced coefficient vector
\[
\phi^0
=
\sum_{j=1}^{10} c_j\psi_j,
\qquad
c=(0.90,-0.80,0.75,-0.70,0.65,-0.60,0.55,-0.50,0.45,-0.40).
\]

The computations use a Fourier spectral $64\times64$ grid. SAV uses the single combined energy \(E_1\), while PB-SAV uses the ten components in \eqref{eq:s6-finite-rank-energy} grouped into \(r\) consecutive blocks of nearly equal size, \(r=2,\ldots,10\); rank \(10\) resolves every component separately.

Figure~\ref{fig:finite-rank-rank-progression} reports the final $L^2$ distance to the ETDRK4 reference at \(T=0.5\). The trajectory error decreases as the correction rank is increased, with rank $10$ giving the smallest error in the displayed balanced sweep.

\begin{figure}[htbp]
  \centering
  \includegraphics[width=0.72\textwidth]{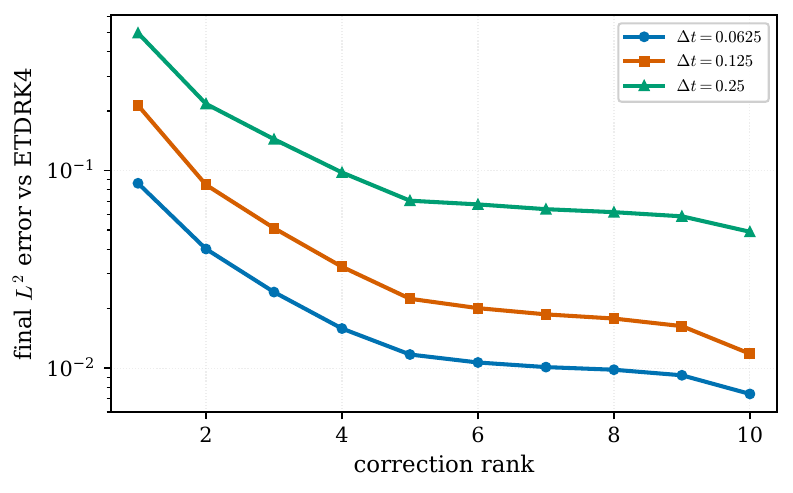}
  \caption{The finite-rank nonlocal model \eqref{eq:s6-finite-rank-pde}--\eqref{eq:s6-finite-rank-energy}: final $L^2$ distance to the ETDRK4 reference at $T=0.5$ as a function of correction rank. The displayed case uses the balanced initial condition.}
  \label{fig:finite-rank-rank-progression}
\end{figure}

The improvement over SAV persists in the sensitivity checks summarized in Table~\ref{tab:finite-rank-robustness}, which uses three coefficient sets for the same ten Fourier components: the \emph{balanced} case from Figure~\ref{fig:finite-rank-rank-progression},
\[
c_{\rm mixed}
=(0.95,-0.35,0.70,-0.55,0.40,-0.30,0.52,-0.44,0.33,-0.25),
\]
and
\[
c_{\rm low}
=(1.05,-0.85,0.45,-0.35,0.25,-0.20,0.15,-0.12,0.10,-0.08).
\]
Each cell reports the final $L^2$ distance of the rank-$10$ scheme divided by that of SAV, so values below one indicate that rank-$10$ PB-SAV is closer to the ETDRK4 reference. The smallest gains occur in the case in which the low modes dominate.

\begin{table}[tbp]
\centering
\small
\caption{Sensitivity checks for the ten-component finite-rank nonlocal model. Each cell reports \((\text{final }L^2\text{ distance of rank-}10\text{ PB-SAV})/(\text{final }L^2\text{ distance of SAV})\) to the same refined ETDRK4 reference. Final time \(T=0.5\).}
\label{tab:finite-rank-robustness}
\begin{tabular}{lccc}
\toprule
case & \(\Delta t=0.0625\) & \(\Delta t=0.125\) & \(\Delta t=0.25\) \\
\midrule
balanced & $0.086$ & $0.055$ & $0.099$ \\
mixed & $0.132$ & $0.015$ & $0.096$ \\
low modes dominant & $0.484$ & $0.275$ & $0.328$ \\
\bottomrule
\end{tabular}
\end{table}

\subsection{Linear nonlocal Cahn--Hilliard equation with exact spectral reference}
\label{subsec:exp-ch-nonlocal}

Our final test is a linear nonlocal Cahn--Hilliard gradient flow with an exact spectral reference solution.  We consider the energy
\begin{equation}
E[\phi]
=
\frac{\varepsilon^2}{2}\int_\Omega|\nabla\phi|^2\,dx
+
\frac{\sigma}{2}\langle\phi-\bar\phi,\mathcal K_\ell(\phi-\bar\phi)\rangle,
\qquad
\mathcal K_\ell=(I-\ell^2\Delta)^{-1},
\label{eq:s6-ch-energy}
\end{equation}
where $\bar\phi$ denotes the spatial mean of $\phi$.
The gradient flow is
\[
\partial_t\phi=\Delta\mu,\qquad \mu=\frac{\delta E}{\delta\phi}.
\]
The gradient term $\tfrac{\varepsilon^2}{2}\int|\nabla\phi|^2\,dx$ is the quadratic part, treated implicitly; the nonlocal term is $E_1[\phi]=(\sigma/2)\langle\phi-\bar\phi,\mathcal K_\ell(\phi-\bar\phi)\rangle$. We use $\varepsilon^2=5\times10^{-4}$, $\sigma=20$, $\ell=0.2$ on a $128\times128$ Fourier spectral grid.

The decomposition $E_1=\sum_{\alpha=1}^r E_{1,\alpha}$ partitions the Fourier shells into $r$ groups. The deterministic broadband initial condition is
\[
\begin{aligned}
\phi^0(x,y)
=\;&
0.16\cos x+0.14\cos(2y)+0.12\cos(4x+3y)+0.10\sin(7x-2y)\\
&+0.08\cos(10x+5y)+0.06\sin(14x-9y)+0.05\cos(18x+11y)\\
&+0.04\sin(21x-13y).
\end{aligned}
\]
The reference is the exact spectral solution of the linear PDE.

Figure~\ref{fig:ch-nonlocal-rank} shows the final $L^2$ error against the exact reference for ranks $r\in\{1,2,4,8\}$ after integration to $T=1$, with $\Delta t\in\{6.25\times10^{-4},1.25\times10^{-3},2.5\times10^{-3},5\times10^{-3}\}$. The pullback correction reduces the trajectory error for all four step sizes, with the largest gain at the coarser steps: at $\Delta t=5\times10^{-3}$ the rank-$1$ error of $1.26\times10^{-1}$ drops to $4.47\times10^{-2}$ at rank $4$ and $4.22\times10^{-2}$ at rank $8$.

\begin{figure}[tbp]
  \centering
  \includegraphics[width=0.68\textwidth]{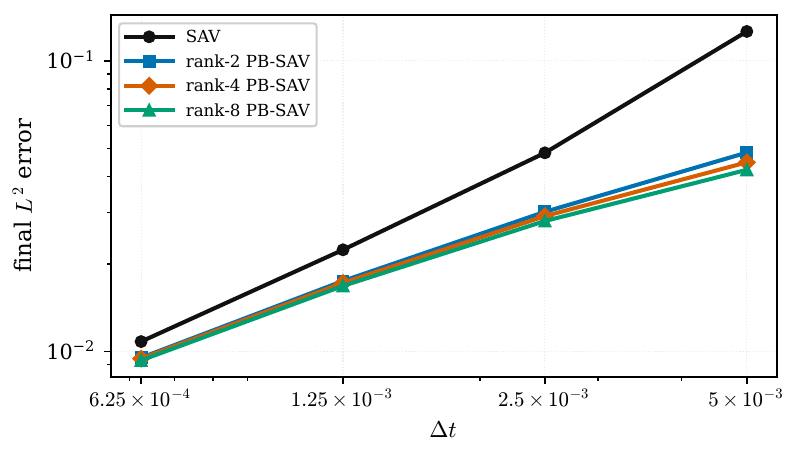}
  \caption{Final $L^2$ error for the linear nonlocal Cahn--Hilliard test at $T=1$, measured against the exact spectral reference.}
  \label{fig:ch-nonlocal-rank}
\end{figure}

\paragraph{Adaptive algorithm}
We next test whether the adaptive grouping rule of
Section~\ref{subsec:variance-guided-grouping} can retain the accuracy of the
fixed fine split while using fewer active components. Starting from a fine partition of
$E_1$ into eight initial-energy bands, we choose a grouped decomposition at each
step by Algorithm~\ref{alg:adaptive-grouping}, with direction $\psi^n$ the
rank-one SAV predictor increment and retention fraction $\tau=0.99$.

Figure~\ref{fig:ch-nonlocal-adaptive} shows that this rule nearly matches the
fixed rank-$8$ trajectory error across all four step sizes while using mean rank
between $2.04$ and $2.73$. At $\Delta t=5\times10^{-3}$ the active rank starts at
$4$, drops to $3$ during the initial transient, and settles at $2$ for most of the
run. 

\begin{figure}[htbp]
  \centering
  \begin{subfigure}[t]{0.48\textwidth}
    \centering
    \includegraphics[width=\textwidth]{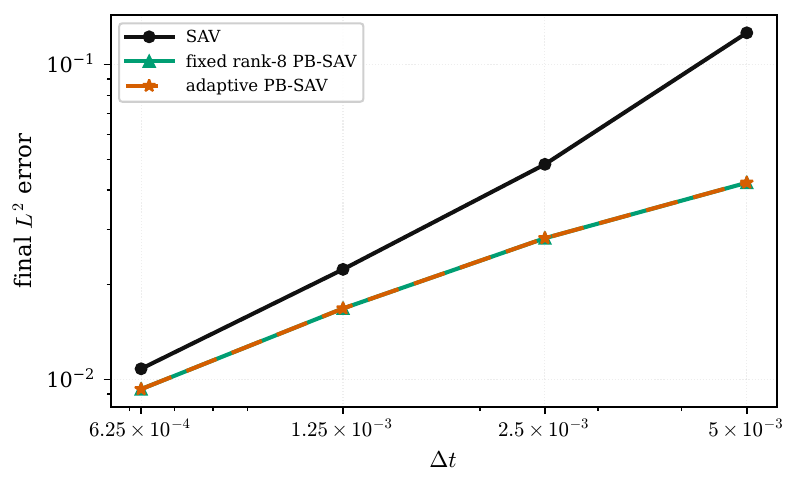}
    \caption{Final $L^2$ error.}
  \end{subfigure}
  \hfill
  \begin{subfigure}[t]{0.48\textwidth}
    \centering
    \includegraphics[width=\textwidth]{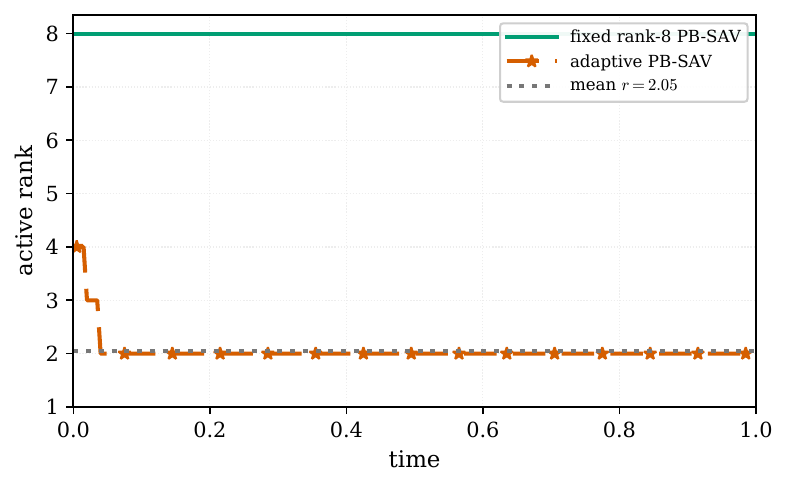}
    \caption{Active rank history.}
  \end{subfigure}
  \caption{Adaptive grouping for the linear nonlocal Cahn--Hilliard test. Panel
  (a) gives final $L^2$ errors at $T=1$ for SAV, fixed rank-$8$ PB-SAV, and the
  adaptive PB-SAV (Algorithm~\ref{alg:adaptive-grouping}) with $\tau=0.99$;
  panel (b) gives the active rank history at $\Delta t=5\times10^{-3}$.}
  \label{fig:ch-nonlocal-adaptive}
\end{figure}

\section{Conclusion}
\label{sec:conclusion}

We have introduced a pullback-corrected SAV (PB-SAV) family for gradient flows.  The method retains the single auxiliary variable of SAV but uses the pullback correction generated by an admissible decomposition of the nonlinear energy.

The SAV energy argument carries over unchanged.  PB-SAV satisfies a modified-energy dissipation law with the same scalar modified energy as SAV; the result also holds for step-dependent decompositions. For a fixed finite-dimensional spatial discretization, the scheme is first-order convergent under standard smoothness and uniform-positivity assumptions.  The correction produced by a refined decomposition dominates the coarser correction in the positive semidefinite order, and the gap is an explicit weighted variance.  This identity explains what refinement changes at the level of the linearized state equation, although the resulting trajectory accuracy still depends on the component forces.  The Sherman--Morrison--Woodbury formula gives an efficient low-rank implementation, and the finite-dimensional viewpoint identifies the correction with the Gauss--Newton matrix of a least-squares representation of the nonlinear energy.

The numerical experiments are consistent with this interpretation.  In some smooth regimes PB-SAV has the same first-order behavior as SAV and differs only in the error constant, which in the smooth Allen--Cahn test is larger than that of SAV.  In finite-dimensional problems with low-rank or grouped curvature, and in nonlocal examples where the component force is well represented by the chosen decomposition, the higher-rank correction substantially improves the trajectory accuracy. 

Future work will address higher-order PB-SAV schemes, systematic adaptive-rank selection, and sharper error estimates that relate the weighted-variance gap to the observed trajectory error.

\section*{Data Availability} Data sharing is not applicable to this article as no datasets were generated or analyzed during the current study.

\section*{Ethics Approval} This study did not involve any human or animal subjects. Therefore It does not require approval from an ethics committee.

\section*{Conflict of interest}
The authors declare that they have no competing interests.
\bibliographystyle{siamplain}
\bibliography{zhang_bibtex}
\end{document}